\documentclass[11pt]{amsart}
\usepackage{amssymb,mathrsfs,graphicx,extpfeil}
\usepackage{epsfig}
\usepackage{indentfirst, latexsym, amssymb, enumerate,amsmath,graphicx}
\usepackage{float}
\usepackage{colortbl}
\usepackage{epsfig,subfigure}
\usepackage{mathtools}
\usepackage{amsmath}
\allowdisplaybreaks[4]
\title[Synchronization in general digraph]{Emergence of synchronization in Kuramoto model with general digraph}

\author[Zhang]{Xiongtao Zhang \textsuperscript{$\dag$,}\textsuperscript{$\S$}}

\author[Zhu]{Tingting Zhu \textsuperscript{$\ddag$,}\textsuperscript{$\S$,}\textsuperscript{*}}

\newtheorem{theorem}{Theorem}[section]
\newtheorem{lemma}{Lemma}[section]
\newtheorem{corollary}{Corollary}[section]

\newtheorem{remark}{Remark}[section]

\newtheorem{definition}{Definition}[section]

\def\charf {\mbox{{\text 1}\kern-.30em {\text l}}}

\begin{document}

\date{\today}


\thanks{$^\dag$ Center for Mathematical Sciences, Huazhong University of Science and Technology, Wuhan, China (xtzhang@hust.edu.cn)}
\thanks{$^\ddag$ School of Mathematics and Statistics, Huazhong University of Science and Technology, Wuhan, China (ttzhud201880016@163.com)}
\thanks{$^\S$ The authors X. Zhang and T. Zhu contribute equally to the work.}
\thanks{$^*$ Corresponding author. }
\thanks{The work of X. Zhang is supported by the National Natural Science Foundation of China (Grant No. 11801194).}

\begin{abstract}
In this paper, we study the complete synchronization of the Kuramoto model with general network containing a spanning tree, when the initial phases are distributed in an open half circle. As lack of uniform coercivity in general digraph, in order to capture the dissipation structure on a general network, we apply the node decomposition criteria in \cite{H-L-Z20} to yield a hierarchical structure, which leads to the hypo-coercivity. This drives the phase diameter into a small region after finite time in a large coupling regime, and the uniform boundedness of the diameter eventually leads to the emergence of exponentially fast synchronization.
\end{abstract}
\maketitle \centerline{\date}

\section{Introduction}\label{sec:1}


Emergent collective behaviors in complex systems are ubiquitous around the world, such as aggregation of bacteria, flocking of birds, synchronous flashing of fireflies and so forth \cite{D-M07,D-M08,De-M08,P-L-S-G-P,T-T,T-B,V-C-B-C-S,Wi}, in which self-propelled agents organize themselves into a particular motion via limited environmental information and simple rules. In order to study the driven mechanism of the emergence of collective behaviors, various dynamic models have been proposed in recent years such as Cucker-Smale model \cite{C-S2}, Kuraomoto model \cite{Ku}, and Winfree model \cite{Wi}, etc.. These seminal models have received lots of attention and have been systematically studied due to their potential applications in biology and engineer, to name a few, modeling of cell and filament orientation, sensor networks, formation control of robots and unmanned aerial vehicles \cite{L-P-L-S,P-L-S-G-P,P-E-G}, etc. 

In the present paper, we focus on the emergence of synchronization in Kuramoto model with general interaction network. The terminology \textit{synchronization} represents the phenomena in which coupled oscillators adjust their rhythms through weak interaction \cite{A-B-V-R-S05, P-R-K}, and Kuramoto model is a classical model to study the emergence of synchronization. The emergent dynamics of the Kuramoto model has been extensively studied in literature \cite{ B-S00, C-H-J-K12,C-S09,D-B11,H-H-K,H-K-K-Z,H-K-L-N21,H-K-P15,H-K-R,H-R20,M-S07,St}.
In our work, to fix the idea, we consider a digraph $\mathcal{G} = (\mathcal{V}, \mathcal{E})$ consisting of a finite set $\mathcal{V} = \{1,\ldots, N\}$ of vertices and a set $\mathcal{E} \subset \mathcal{V} \times \mathcal{V}$ of directed arcs. We assume that Kuramoto oscillators are located at vertices and interact with each other via the underlying network topology. For each vertex $i$, we denote the set of its neighbors by $\mathcal{N}_i$, which is the set of vertices that directly influence vertex $i$. Now, let $\theta_i = \theta_i(t)$ be the phase of the Kuramoto oscillator at vertex $i$, and define the $(0,1)$-adjacency matrix $(\chi_{ij})$ as follows:
\begin{equation*}
\chi_{ij} = 
\begin{cases}
\displaystyle 1 \quad \mbox{if the $j$th oscillator influences the $i$th oscillator}, \\
\displaystyle 0 \quad \mbox{otherwise}.
\end{cases}
\end{equation*}
Then, the set of neighbors of $i$-th oscillator is actually $\mathcal{N}_i := \{j: \chi_{ij} >0\}$.
In this setting, the dynamics of phase $\theta_i$ is governed by the following ordinary differential system:
\begin{equation}\label{KuM}
\begin{cases}
\displaystyle \dot{\theta}_i(t) = \Omega_i + \kappa \underset{j \in \mathcal{N}_i}{\sum} \sin(\theta_j(t) - \theta_i(t)), \quad t>0, \quad i \in \mathcal{V}, \\
\displaystyle \theta_i(0) = \theta_{i0}.
\end{cases}
\end{equation}
where $\kappa > 0$ is the uniform coupling strength and $\Omega_i$ represents the intrinsic natural frequency of the $i$th oscillator drawn from some distribution function $g=g(\Omega)$.
The motivation to consider general network is very natural, since the non all-to-all or non-symmetric interactions are common in the real world. For instance, flying birds can make a flocking cluster via the influence from several neighbors, while the sheep can form a group by following the leader. Therefore, study on the dynamical system on a general digraph is natural and important, and gradually attracts a lot of researchers from different areas. We refer the readers to the following references for more details of the background \cite{C-L18, D-H-K19, D-H-K20, D-X, H-KD18, H-K-P-R16,   H-L-X, H-L-Z20}.

There are few works \cite{D-H-K19, D-X, H-L-X} on the synchronization of the Kuramoto model on a general digraph in contrast with the complete graph. More precisely, 
the authors in \cite{D-X} studied the generalized Kuramoto model with directed coupling topology, which is allowed to be non-symmetric. They showed the frequency synchronization when the initial phases of oscillators are distributed over the open half circle for a large class of coupling structure.  
However, they required any pair of oscillators have one common neighbor, so that the dissipation structure can be captured by the good property of sine function. In \cite{H-L-X}, the authors provided an asymptotic formation of phase-locked states for the ensemble of Kuramoto oscillators with a symmetric and connected network, when the initial configuration is distributed in a half circle. More precisely, they exploit the gradient structure and use energy method to derive complete synchronization whereas there is no information about the convergence rate. In literature \cite{D-H-K20}, the authors studied a network structure containing a spanning tree (see Definition \ref{spanning_tree}) on the collective behaviors of Kuramoto oscillators. Actually, they lift the Kuramoto model to second-order system such that the second-order formulation enjoys several similar mathematical structures as for the Cucker-Smale flocking model \cite{D-H-K19}. But this method only works when initial phases are confined in a quarter circle, since the cosine function becomes negative if $\frac{\pi}{2}< \theta<\pi$.

So far, if the ensemble distributed in half circle, the dissipation structure of the Kuramoto model with general digragh is still unclear. The main difficulty is that, when considering the ensemble in half circle, there is no uniform coercive inequality to yield the dissipation, which is due to the non-all-to-all and non-symmetric structure. For example, the time derivative of the diameter may be zero in the general digraph. Therefore, we switch to construct the hypo-coercivity similar as in \cite{H-L-Z20}, which will help us to capture the dissipation structure. Comparing to \cite{H-L-Z20} which deals with the Cucker-Smale model on a general digraph, the interactions in Kuramoto model lack the monotonic property since $\sin(x)$ is not monotonic in half circle. Therefore, we choose more delicate constructions and estimates of the convex combinations to fit the special structure of Kuramoto model, which eventually yields the following main theorem. 

\begin{theorem}\label{enter_small}
Suppose that the network topology $(\chi_{ij})$ contains a spanning tree, and let $\theta_i$ be a solution to \eqref{KuM}. Moreover, assume that the initial data and the quantity $\eta$ satisfy
 \begin{equation}\label{condition_1}
 D(\theta(0)) < \alpha < \gamma < \pi, \quad \eta > \max \left\{\frac{1}{\sin \gamma}, \frac{2}{1 - \frac{\alpha}{\gamma}} \right\},
 \end{equation}
 where $\alpha, \gamma$ are constants. Then, we can find a sufficiently small positive constant $D^\infty < \min\left\{\alpha,\frac{\pi}{2}\right\}$ and a corresponding time $t_*$ such that 
  \[D(\theta(t)) < D^\infty, \qquad t \in (t_*, + \infty),\]
 provided the coupling strength $\kappa$ satisfies
 \begin{equation}\label{condition_2}
 \kappa > \left( 1 + \frac{(d+1)\alpha}{\alpha - D(\theta(0))}\right) \frac{(4c)^d \tilde{c}}{\beta^{d+1}D^\infty},
 \end{equation}
 where $d$ is the number of general nodes which is smaller than $N$ (see Section \ref{sec:2}) and
\[c = \frac{(2N + 1)(\sum_{j=1}^{N-1} \eta^j A(2N,j) + 1)\gamma}{\sin \gamma}, \quad \tilde{c} = \frac{D(\Omega) (\sum_{j=1}^{N-1}\eta^j A(2N,j) + 1)\gamma}{\sin \gamma}.\]
\end{theorem}

Note that Theorem \ref{enter_small} only shows the small and uniform boundedness of the ensemble, then we can directly apply the methods and results in \cite{H-L-Z20} or \cite{D-H-K20} to yield the exponentially fast emergence of frequency synchronization. Therefore, we will only show the detailed proof of Theorem \ref{enter_small}.
 
The rest of the paper is organized as follows. In Section \ref{sec:2}, we recall some concepts on the network topology and provide an a priori local-in-time estimate about phase diameter of the ensemble. In Section \ref{sec:3},  we consider a strong connected ensemble for which the initial phases are distributed in the open half circle. We show that the phase diameter is uniformly bounded and will be confined in a small region after some finite time in a large coupling regime. In Section \ref{sec:4}, we study the general network with a spanning tree structure. In our framework, the coupling strength is sufficiently large and the initial data is confined in an open half circle. We use the inductive argument and show that the phase diameter of the whole digraph will concentrate into a small region of a quarter circle after some finite time, which yields the exponential emergence of synchronization. Section \ref{sec:5} is devoted to a brief summary.\newline

\section{Preliminaries}\label{sec:2}
\setcounter{equation}{0}
In this section, we introduce some basic concepts such as spanning tree and node decomposition of a general network \eqref{KuM}. Then, we will provide some necessary notations and an a priori estimate that will be frequently used in later sections. 
 
\subsection{Spanning tree}
Roughly speaking, spanning tree means we can find an oscillator which affects all the other oscillators directly or in-directly. In other words, a system without spanning tree can be separated into two parts without any interactions. Therefore, this is the most important structure for emergence of collective behavior on a general digraph. 

More precisely, let the network topology be registered by the neighbor set $\mathcal{N}_i$ which consists of all neighbors of the $i$th oscillator. Then, for a given set of $\{\mathcal{N}_i\}_{i=1}^N$ in system \eqref{KuM}, we have the following definition.

\begin{definition}\label{spanning_tree}
\begin{enumerate}[(1)]
\item The Kuramoto digraph $\mathcal{G} = (\mathcal{V}, \mathcal{E})$ associated to \eqref{KuM} consists of a finite set $\mathcal{V} = \{1,2,\ldots,N\}$ of vertices, and a set $\mathcal{E} \subset \mathcal{V} \times \mathcal{V}$ of arcs with ordered pair $(j,i) \in \mathcal{E}$ if $j \in \mathcal{N}_i$.

\item A path in $\mathcal{G}$ from $i_1$ to $i_k$ is a sequence $i_1,i_2,\ldots,i_k$ such that
\[i_s \in \mathcal{N}_{i_{s+1}} \quad \mbox{for} \ 1 \le s \le k-1.\]
If there exists a path from $j$ to $i$, then vertex $i$ is said to be reachable from vertex $j$.

\item The Kuramoto digraph contains a spanning tree if we can find a vertex such that any other vertex of $\mathcal{G}$ is reachable from it.
\end{enumerate}
\end{definition}

\noindent According to the discussion of spanning tree in the beginning of this part, in order to guarantee the emergence of synchronization, we will always assume the existence of a spanning tree throughout the paper. Now we recall the concepts of root and general root in \cite{H-L-Z20}. Let $l,k \in \mathbb{N}$ with $1 \le l \le k \le N$, and let $C_{l,k} = (c_l,c_{l+1},\ldots,c_k)$ be a vector in $\mathbb{R}^{k-l+1}$ such that
\[c_i \ge 0, \quad l \le i \le k \quad \mbox{and} \quad \sum_{i=l}^k c_i = 1.\]
For an ensembel of $N$-oscillators with phase $\{\theta_i\}_{i=1}^N$, we set $\mathcal{L}_l^k(C_{l,k})$ to be a convex combination of $\{\theta_i\}_{i=l}^k$ with the coefficient $C_{l,k}$:
\[\mathcal{L}_l^k(C_{l,k}) := \sum_{i=l}^k c_i \theta_i.\]
Note that each $\theta_i$ is a convex combination of itself, and particularly $\theta_N = \mathcal{L}_N^N(1)$ and $\theta_1 = \mathcal{L}_1^1(1)$.

\begin{definition}\label{general_root}
(Root and general root)
\begin{enumerate}
\item We say $\theta_k$ is a root if it is not affected by the rest oscillators, i.e., $j \notin \mathcal{N}_k$ for any $j \in \{1,2,\ldots,N\} \setminus \{k\}$.

\item We say $\mathcal{L}_l^k(C_{l,k})$ is a general root if $\mathcal{L}_l^k(C_{l,k})$ is not affected by the rest oscillators, i.e., for any $i\in \{l, l+1,\ldots,k\}$ and $j \in \{1,2,\ldots,N\} \setminus \{l,l+1,\ldots, k\}$, we have $j \notin \mathcal{N}_i$. 
\end{enumerate}
\end{definition}

\begin{lemma}\label{ }\cite{H-L-Z20}
The following assertions hold.
\begin{enumerate}
\item If the network contains a spanning tree, then there is at most one root.
\item Assume the network contains a spanning tree. If $\mathcal{L}_k^N(C_{k,N})$ is a general root, then $\mathcal{L}_1^l(C_{1,l})$ is not a general root for each $l \in \{1,2,\ldots,k-1\}$.
\end{enumerate}
\end{lemma}

\subsection{Node decomposition}
In this subsection, we will introduce the concept of maximum node. Then, we can introduce node decomposition to represent the whole graph $\mathcal{G}$ (or say vertex set $\mathcal{V}$) as a disjoint union of a sequence of nodes. The key point is that the node decomposition shows a hierarchical structure which allows us to apply the induction principle. Let $\mathcal{G} = (\mathcal{V}, \mathcal{E})$ and $\mathcal{V}_1 \subset \mathcal{V}$, a subgraph $\mathcal{G}_1 = (\mathcal{V}_1, \mathcal{E}_1)$ is the digraph with  vertex set $\mathcal{V}_1$ and arc set $\mathcal{E}_1$ which consists of the arcs in $\mathcal{G}$ connecting agents in $\mathcal{V}_1$. For convenience, for a given digraph $\mathcal{G} = (\mathcal{V}, \mathcal{E})$, we will identify a subgraph $\mathcal{G}_1 = (\mathcal{V}_1, \mathcal{E}_1)$ with its vertex set $\mathcal{V}_1$. Now we first introduce the definition of nodes below.

\begin{definition}\label{node} (Node)
Let $\mathcal{G}$ be a digraph. A subset $\mathcal{G}_1$ of vertices is called a node if it is strongly connected, i.e., for any subset $\mathcal{G}_2$ of $\mathcal{G}_1$, $\mathcal{G}_2$ is affected by $\mathcal{G}_1\setminus\mathcal{G}_2$. Moreover, if $\mathcal{G}_1$ is not affected by $\mathcal{G}\setminus\mathcal{G}_1$, we say $\mathcal{G}_1$ is a maximum node. 
\end{definition}

Notably, a node can be understood intuitively in a manner that a set of oscillators can be viewed as a "large" oscillator. Next, we can exploit the concept of node to simplify the structure of the digraph, and this can help us to catch the attraction effect more clearly in the underlying network topology.

\begin{lemma}\label{one_maximum}\cite{H-L-Z20}
Any digraph $\mathcal{G}$ contains at least one maximum node. A digraph $\mathcal{G}$ contains a unique maximum node if and only if $\mathcal{G}$ has a spanning tree.
\end{lemma}

\begin{lemma}\label{Node decomposition}\cite{H-L-Z20}(Node decomposition)
Let $\mathcal{G}$ be any digraph. Then we can decompose $\mathcal{G}$ to be a union as $\mathcal{G} = \bigcup_{i=0}^d (\bigcup_{j=1}^{k_i} \mathcal{G}_i^j)$ such that
\begin{enumerate}
\item $\mathcal{G}_0^j$ are the maximum nodes of $\mathcal{G}$, where $1\le j\le k_0$.
\item For any $p,q$ where $1 \le p\le d$ and $1\le q \le k_p$, $\mathcal{G}_p^q$ are the maximum nodes of $\mathcal{G} \setminus (\bigcup_{i=0}^{p-1} (\bigcup_{j=1}^{k_i} \mathcal{G}_i^j))$.
\end{enumerate}
\end{lemma}

\begin{remark}\label{w.l.o.g.} Lemma \ref{Node decomposition} shows a clear hierarchical structure on a general digraph. For the convenience of later analysis, we give some comments on important notations and properties to be used throughout the paper.
\begin{enumerate}
\item According to the definition of maximum node, we know $\mathcal{G}_p^q$ and $\mathcal{G}_p^{q'}$ do not influence each other for $1\le q \ne q' \le k_p$. Actually, $\mathcal{G}_p^q$ will only be affected by $\mathcal{G}_0$ and $\mathcal{G}_i^j$, where $1 \le i \le p-1, \ 1 \le j \le k_i$. Therefore without loss of generality, we may assume $k_i = 1$ for all $1 \le i \le d$ in the proof of our main theorem (see Theorem \ref{enter_small}). Thus, the decomposition can be expressed by
\[\mathcal{G} = \bigcup_{i=0}^d \mathcal{G}_i,\]
where $\mathcal{G}_p$ is a maximum node of $\mathcal{G}\setminus (\bigcup_{i=0}^{p-1} \mathcal{G}_i)$.

\item Given an oscillator $\theta_i^{k+1} \in \mathcal{G}_{k+1}$, we denote by $\bigcup_{j=0}^{k+1} \mathcal{N}_i^{k+1}(j)$ the set of neighbors of $\theta_i^{k+1}$, where $\mathcal{N}_i^{k+1}(j)$ represents the neighbors of $\theta_i^{k+1}$ in $\mathcal{G}_j$. The node decomposition and spanning tree structure in $\mathcal{G}$ guarantee that $\bigcup_{j=0}^k \mathcal{N}_p^{k+1}(j) \ne \emptyset$.
\end{enumerate}
\end{remark}

\subsection{Notations and local estimates} In this part, for simplicity, we introduce some notations, such as the extreme phase, phase diameter of $\mathcal{G}$ and the first $k+1$ nodes, frequency diameter, and cardinality of subdigraph:
\begin{align*}
&\theta_M =\max_{1\le k \le N} \{\theta_k\} = \max_{0\le i \le d} \max_{1 \le j \le N_i} \{\theta_j^i\}, \quad \theta_m = \min_{1\le k \le N} \{\theta_k\} = \min_{0\le i \le d} \min_{1 \le j \le N_i} \{\theta_j^i\},\\
&D(\theta) = \theta_M - \theta_m,\quad D_k(\theta) = \max_{0 \le i \le k} \max_{1 \le j \le N_i} \{\theta_j^i\} - \min_{0 \le i \le k} \min_{1 \le j \le N_i} \{\theta_j^i\}, \\
&\Omega_M =\max_{0\le i \le d} \max_{1 \le j \le N_i} \{\Omega_j^i\}, \quad \Omega_m =\min_{0\le i \le d} \min_{1 \le j \le N_i} \{\Omega_j^i\}, \quad D(\Omega) = \Omega_M - \Omega_m,\\
&N_i = |\mathcal{G}_i|, \quad S_k = \sum_{i=0}^k N_i, \quad 0 \le k \le d, \quad \sum_{i=0}^d N_i = N.
\end{align*}
Finally, we provide an a priori local-in-time estimate on the phase diameter to finish the section, which shows the diameter of the ensemble remains less than $\pi$ in short time.

\begin{lemma}\label{diameter_alpha}
Let $\theta_i$ be a solution to system \eqref{KuM} and suppose the initial phase diameter satisfies $D(\theta(0)) < \alpha < \gamma < \pi$. Then there exists time $\bar{t}$ such that
\begin{equation}\label{t_bar}
D(\theta(t)) < \alpha, \quad \forall \ t \in [0,\bar{t}),
\end{equation}
where $\alpha,\gamma$ are constants and $\bar{t} = \frac{\alpha - D(\theta(0))}{D(\Omega)}$.
\end{lemma}
\begin{proof}
According to system \eqref{KuM}, we have
\begin{equation*}
\dot{\theta}_M = \Omega_M + \kappa \sum_{j \in \mathcal{N}_M} \sin(\theta_j - \theta_M), \qquad \dot{\theta}_m = \Omega_m + \kappa \sum_{j \in \mathcal{N}_m} \sin(\theta_j - \theta_m).
\end{equation*}
When the phase diameter is located in $[D(\theta(0)), \alpha]$, it is obvious that
\begin{equation*}
\sum_{j \in \mathcal{N}_M} \sin(\theta_j - \theta_M) \le 0, \quad \sum_{j \in \mathcal{N}_m} \sin(\theta_j - \theta_m) \ge 0.
\end{equation*}
Hence, the dynamics of phase diameter of all nodes can be estimated as follows
\begin{equation}\label{G-1}
\dot{D}(\theta(t)) = \frac{d}{dt} (\theta_M - \theta_m) \le D(\Omega).
\end{equation}
That is to say, the growth of phase diameter is less than the linear growth with slope $D(\Omega)$ if $D(\theta(t)) \in [D(\theta(0)), \alpha]$. Set $\bar{t} = \frac{\alpha - D(\theta(0))}{D(\Omega)}$. Then according to \eqref{G-1}, it can be seen that $D(\theta(t))$ is less than $\alpha$ before time $\bar{t}$, i.e.,
\begin{equation*}
D(\theta(t)) < \alpha, \quad \forall \ t \in [0,\bar{t}),
\end{equation*}
\end{proof}

\section{Strong connected case}\label{sec:3}
\setcounter{equation}{0}
We will first study the special case when the network is strongly connected. Without loss of generality, we denote the strong connected graph by $\mathcal{G}_0$. According to Definition \ref{node}, Lemma \ref{one_maximum} and Lemma \ref{Node decomposition}, this means the network contains only one maximum node. Then, we will show the emergence of complete synchronization in the strong connected case. We now introduce an algorithm to construct a proper convex combination of the oscillators, which can involve the dissipation from interaction of general network. More precisely, the algorithm for $\mathcal{G}_0$ consists of the following three steps:\newline

\noindent\textbf{Step 1.} For any given time $t$, we reorder the oscillator indexes to make the oscillator phases from minimum to maximum. More specifically, by relabeling the agents at time $t$, we set
\begin{equation}\label{well_order}
\theta_1^0(t) \le \theta_2^0(t) \le \ldots \le \theta_{N_0}^0(t).
\end{equation}
In order to introduce the following steps, we first provide the process of iterations for $\bar{\mathcal{L}}_k^{N_0}(\bar{C}_{k,N_0})$ and $\underline{\mathcal{L}}_1^l(\underline{C}_{1,l})$ as follows:\newline

\noindent $\bullet$($\mathcal{A}_1$): If $\bar{\mathcal{L}}_k^{N_0}(\bar{C}_{k,N_0})$ is not a general root, then we construct
\[\bar{\mathcal{L}}_{k-1}^{N_0}(\bar{C}_{k-1,N_0}) = \frac{\bar{a}_{k-1} \bar{\mathcal{L}}_k^{N_0}(\bar{C}_{k,N_0}) + \theta^0_{k-1}}{\bar{a}_{k-1} + 1}.\]

\noindent $\bullet$($\mathcal{A}_2$): If $\underline{\mathcal{L}}_1^l(\underline{C}_{1,l})$ is not a general root, then we construct
\[\underline{\mathcal{L}}_1^{l+1}(\underline{C}_{1,l+1}) = \frac{\underline{a}_{l+1} \underline{\mathcal{L}}_1^l(\underline{C}_{1,l}) + \theta^0_{l+1}}{\underline{a}_{l+1} + 1}\]

\noindent\textbf{Step 2.} According to the strong connectivity of $\mathcal{G}_0$, we immediately know that $\bar{\mathcal{L}}_1^{N_0}(\bar{C}_{1,N_0})$ is a general root, and $\bar{\mathcal{L}}_k^{N_0}(\bar{C}_{k,N_0})$ is not a general root for $k >1$. Therefore, we may start from $\theta^0_{N_0}$ and follow the process $\mathcal{A}_1$ to construct $\bar{\mathcal{L}}_k^{N_0}(\bar{C}_{k,N_0})$ until $k=1$.\newline

\noindent\textbf{Step 3.} Similarly, we know that $\underline{\mathcal{L}}_1^{N_0}(\underline{C}_{1,N_0})$ is a general root and $\underline{\mathcal{L}}_1^l(\underline{C}_{1,l})$ is not a general root for $l < N_0$. Therefore, we may start from $\theta^0_1$ and follow the process $\mathcal{A}_2$ until $l=N_0$.\newline  

We emphasize that the order of the oscillators will change along time $t$, but at each time $t$, the above algorithm works. For convenience, the algorithm from Step 1 to Step 3 will be referred as Algorithm $\mathcal{A}$. Then, according to Algorithm $\mathcal{A}$, we will show a monotone property about the function $\sin x$, and provide a priori estimates which will be crucially used later in the proof of uniform boundness of phase diameter.

\begin{lemma}\label{eta_sin_inequality}
Let $\theta_i = \{\theta^0_i\}$ be a solution to system \eqref{KuM} with srong connected network $\mathcal{G}_0$.  Moreover at time $t$, for the digraph $\mathcal{G}_0$, we also assume that the oscillators are well-ordered as \eqref{well_order}, the phase diameter and the quantity $\eta$ satisfiy the following condition:
\begin{equation*}
D_0(\theta(t)) < \gamma, \quad \eta > \max \left\{\frac{1}{\sin \gamma}, \frac{2}{1 - \frac{\alpha}{\gamma}} \right\},
\end{equation*}
where $\alpha, \gamma$ are given in the condition \eqref{condition_1}.
Then at time $t$, we have
 \begin{equation*}
 \begin{cases}
 \displaystyle \sum_{i=n}^{N_0} ( \eta^{i-n} \underset{j\le i}{\min_{j \in \mathcal{N}^0_i(0)}} \sin (\theta_j^0 - \theta_i^0)) \le \sin(\theta^0_{\bar{k}_n} - \theta^0_{N_0}), \quad \bar{k}_n = \min_{j \in \cup_{i=n}^{N_0} \mathcal{N}^0_i(0)} j, \quad 1\le n \le N_0. \\
 \displaystyle \sum_{i=1}^n (\eta^{n-i} \underset{j \ge i}{\max_{j \in \mathcal{N}^0_i(0)}} \sin (\theta_j^0 - \theta_i^0)) \ge \sin (\theta^0_{\underline{k}_n} - \theta_1^0), \quad \underline{k}_n = \max_{j \in \cup_{i=1}^n \mathcal{N}^0_i(0)} j, \quad 1 \le n \le N_0.
 \end{cases}
 \end{equation*}
\end{lemma}

\begin{proof}
We will only prove the first inequality, the second relation can be proved in a similar manner. In fact, if $N_0=1$, i.e., $N_0$ is a (general) root, we are done. Now we consider the case $N_0 \ge 2$. Due to the strong connectivity of the digraph $\mathcal{G}_0$, $\bar{\mathcal{L}}_1^{N_0}(\bar{C}_{1,N_0})$ is a general root while $\bar{\mathcal{L}}_k^{N_0}(\bar{C}_{k,N_0})$ is not a general root for $k >1$. 

For any given $n \in [1,N_0]$,  we have $\bar{k}_n = \min\limits_{j \in \cup_{i=n}^{N_0} \mathcal{N}^0_i(0)} j$. Hence, there exists $l_0 \in [n,N_0]$ such that $\bar{k}_n \in \mathcal{N}^0_{l_0}(0)$ due to the fact  $\bar{k}_n \in \cup_{i=n}^{N_0} \mathcal{N}^0_i(0)$. For $l_0$, since $\bar{\mathcal{L}}_{l_0+1}^{N_0}(\bar{C}_{l_0+1,N_0})$ is not a general root, there exist $j_0 \le l_0$ and $l_1 \in [l_0+1, N_0]$ such that $j_0 \in \mathcal{N}^0_{l_1}(0)$. For $l_1$, as $\bar{\mathcal{L}}_{l_1+1}^{N_0}(\bar{C}_{l_1+1,N_0})$ is not a general root, there exist $j_1 \le l_1$ and $l_2 \in [l_1+1, N_0]$ such that $j_1 \in \mathcal{N}_{l_2}^0(0)$. we repeat the process until find some $l_p =N_0$ and $j_{p-1} \le l_{p-1}$ such that $j_{p-1} \in \mathcal{N}^0_{l_p}(0) = \mathcal{N}_{N_0}^0(0)$.
Obviously, we have
\begin{equation}\label{A-1}
\begin{aligned}
\sum_{i=n}^{N_0} ( \eta^{i-n} \underset{j\le i}{\min_{j \in \mathcal{N}^0_i(0)}} \sin (\theta_j^0 - \theta_i^0)) &\le \overbrace{\eta^{N_0-n} \sin (\theta_{j_{p-1}}^0 -\theta^0_{N_0}) + \eta^{l_{p-1} -n} \sin (\theta^0_{j_{p-2}} - \theta^0_{l_{p-1}})}^{\mathcal{I}_1} \\
&+  \eta^{l_{p-2} -n} \sin (\theta^0_{j_{p-3}} - \theta^0_{l_{p-2}}) + \cdots + \eta^{l_{2} -n} \sin (\theta^0_{j_{1}} - \theta^0_{l_{2}}) \\
&+\eta^{l_{1} -n} \sin (\theta^0_{j_{0}} - \theta^0_{l_{1}}) + \eta^{l_{0} -n} \sin (\theta^0_{\bar{k}_n}} - \theta^0_{l_{0}).
\end{aligned}
\end{equation}
where we have the following relations
\begin{equation*}
j_k \le l_k, \quad l_k < l_{k+1}, \quad k =0,1,\ldots,p-1.
\end{equation*}
 In the following, we plan to add all the terms on the right-hand side of \eqref{A-1} together to yield the desired estimate. We only consider the case $\gamma > \frac{\pi}{2}$, and the situation $\gamma \le \frac{\pi}{2}$ can be similarly dealt with. We first deal with $\mathcal{I}_1$ in \eqref{A-1}. If $\theta^0_N - \theta^0_{j_{p-1}} \le \frac{\pi}{2}$, we obtain that $0 < \theta^0_N - \theta^0_{l_{p-1}} \le \theta^0_N - \theta^0_{j_{p-1}} \le \frac{\pi}{2}$ due to $j_{p-1} \le l_{p-1}$. Hence, according to $l_{p-1} < N_0$, the following assertion can be obtained
 \begin{equation}\label{3-3}
 \eta^{N_0-n} \sin (\theta^0_{j_{p-1}} -\theta^0_N) \le \eta^{l_{p-1}-n} \sin (\theta^0_{j_{p-1}} -\theta^0_N) \le \eta^{l_{p-1}-n} \sin (\theta^0_{l_{p-1}} -\theta^0_N),
 \end{equation}
On the other hand, if $\frac{\pi}{2} < \theta^0_N - \theta^0_{j_{p-1}} \le D_0(\theta(t)) < \gamma$. It's clear that $\sin (\theta^0_N - \theta^0_{j_{p-1}}) > \sin \gamma$. Then according to the strict inequality $l_{p-1} < N_0$ and $\eta > \frac{1}{\sin \gamma} \ge 1$, we can obtain that
 \begin{equation}\label{3-4}
 \begin{aligned}
\eta^{N_0-n} \sin (\theta^0_{j_{p-1}} -\theta^0_{N_0}) \le -  \eta^{N_0-n-1} \eta \sin \gamma \le - \eta^{N_0-n-1} \le -\eta^{l_{p-1}-n} \le \eta^{l_{p-1}-n} \sin (\theta^0_{l_{p-1}} -\theta^0_N),
 \end{aligned}
 \end{equation}
 where the last inequality holds due to the fact $\sin x \ge -1$. Therefore, combining above estimates \eqref{3-3} and \eqref{3-4}, we obtain that 
 \begin{equation}\label{A-2}
 \eta^{N_0-n} \sin (\theta^0_{j_{p-1}} -\theta^0_{N_0}) \le \eta^{l_{p-1}-n} \sin (\theta^0_{l_{p-1}} -\theta^0_{N_0}).
 \end{equation}
  Next, we apply \eqref{A-2} and the concave property of $\sin x$ in half circle to estimate the term $\mathcal{I}_1$ as follows: 
 \begin{equation}\label{A-3}
 \mathcal{I}_1 \le \eta^{l_{p-1}-n} \sin (\theta^0_{l_{p-1}} -\theta^0_N) + \eta^{l_{p-1} -n} \sin (\theta^0_{j_{p-2}} - \theta^0_{l_{p-1}}) \le \eta^{l_{p-1}-n} \sin(\theta^0_{j_{p-2}} - \theta^0_N).
 \end{equation}
 Finally, we repeat the similar argument in \eqref{A-2} and \eqref{A-3} to obtain that
\begin{equation*}
\begin{aligned}
\sum_{i=n}^{N_0} ( \eta^{i-n} \underset{j\le i}{\min_{j \in \mathcal{N}^0_i(0)}} \sin (\theta^0_j - \theta^0_i)) &\le \eta^{l_0-n} \sin(\theta^0_{l_0}- \theta^0_{N_0}) + \eta^{l_0 -n} \sin (\theta^0_{\bar{k}_n} - \theta^0_{l_0})\\
& \le \eta^{l_0 -n} \sin (\theta^0_{\bar{k}_n} - \theta^0_{N_0}) \le \sin (\theta^0_{\bar{k}_n} - \theta^0_{N_0}),
\end{aligned}
\end{equation*}
where the last inequality holds since $l_0 \ge n$. Therefore we derive the desired result.
\end{proof} 

Based on a priori estimates in Lemma \ref{eta_sin_inequality}, we next design a proper convex combination so that we can capture the dissipation structure. Recall the strongly connected ensemble $\mathcal{G}_0$, and denote by $\theta_i^0\  (i=1,2,\ldots,N_0)$ the members in $\mathcal{G}_0$. Now we assume that at time $t$, the oscillators in $\mathcal{G}_0$ are well-ordered as follows, 
\begin{equation*}
\theta^0_1(t) \le \theta^0_2(t) \le \ldots \le \theta^0_{N_0}(t).
\end{equation*}
Then we apply the process $\mathcal{A}_1$ from $\theta^0_{N_0}$ to $\theta^0_1$ and the process $\mathcal{A}_2$ from $\theta^0_1$ to $\theta^0_{N_0}$ to respectively construct 
\begin{equation}\label{coeffi_0}
\begin{aligned}
&\bar{\mathcal{L}}_{k-1}^{N_0}(\bar{C}_{k-1,N_0}) \ \mbox{with} \ \bar{a}^0_{N_0} = 0, \ \bar{a}^0_{k-1} =\eta (2N_0 -k +2)(\bar{a}^0_k + 1), \quad 2 \le k \le N_0,\\
&\underline{\mathcal{L}}_1^{k+1}(\underline{C}_{1,k+1}) \ \mbox{with} \ \underline{a}^0_1 = 0, \ \underline{a}^0_{k+1} = \eta(k+1+N_0)(\underline{a}^0_k + 1), \quad 1\le k \le N_0-1,
\end{aligned}
\end{equation}
where $N_0$ is the cardinality of $\mathcal{G}_0$ and $\eta$ is given in the condition \eqref{condition_1}. By induction, we can derive explict expressions about the constructed coefficients:
\begin{equation}\label{permutation_0}
\begin{aligned}
&\bar{a}^0_{k-1} = \sum_{j=1}^{N_0 - k+1} \eta^j A(2N_0-k+2,j), \quad 2 \le k \le N_0, \\
&\underline{a}^0_{k+1} = \sum_{j=1}^k \eta^j A(k+1+ N_0,j), \quad 1 \le k \le N_0-1.
\end{aligned}
\end{equation}
Note that $\bar{a}^0_{N_0+1-i} = \underline{a}^0_i, \ i=1,2\ldots,N_0$. And we set
\begin{equation}\label{abbreviation_0}
\bar{\theta}^0_k := \bar{\mathcal{L}}_k^{N_0}(\bar{C}_{k,N_0}),\quad \underline{\theta}^0_k := \underline{\mathcal{L}}_1^k(\underline{C}_{1,k}), \quad 1 \le k \le N_0.
\end{equation}
We define $Q^0 = \bar{\theta}_0 - \underline{\theta}_0$ where $\bar{\theta}_0 = \bar{\theta}^0_1$ and $\underline{\theta}_0 = \underline{\theta}^0_{N_0}$. Note that $Q^0(t)$ is Lipschitz continuous with respect to $t$.
We then establish the comparison relation between $Q^0$ and the phase diameter $D_0(\theta)$ of  $\mathcal{G}_0$ in the following lemma.

\begin{lemma}\label{beta_QD}
Let $\theta_i=\{\theta^0_i\}$ be a solution to system \eqref{KuM} with strong connected digraph $\mathcal{G}_0$. Assume that for the group $\mathcal{G}_0$, the coefficients $\bar{a}_k^0$'s and $\underline{a}_k^0$'s satisfy the scheme \eqref{coeffi_0}. Then at each time $t$, we have the following relation
\begin{equation*}
\beta D_0(\theta(t)) \le Q^0(t) \le D_0(\theta(t)),  \quad \beta = 1 - \frac{2}{\eta},
\end{equation*}
where $\eta$ satisfies the condition \eqref{condition_1}. 
\end{lemma}
\begin{proof}
From the convex combination structure of $\bar{\theta}_0$ and $\underline{\theta}_0$, we immediately have
\begin{equation*}
Q^0(t) = \bar{\theta}_0 - \underline{\theta}_0 \le \theta^0_{N_0}(t) - \theta^0_1(t) = D_0(\theta(t)).
\end{equation*}
We now prove the left part of the desired relation. In fact, we have the following estimate about $Q^0(t)$:
\begin{equation}\label{QD_01}
\begin{aligned}
Q^0(t) &= \bar{\theta}_0(t) - \underline{\theta}_0(t) = \bar{\theta}_0(t) - \theta_{N_0}^0(t) + \theta_{N_0}^0(t) - \theta_1^0(t) + \theta^0_1(t) - \underline{\theta}_0(t) \\
&= \theta_{N_0}^0(t) - \theta_1^0(t) + \bar{\theta}_0(t) - \theta_{N_0}^0(t) + \theta^0_1(t) - \underline{\theta}_0(t) \\
& = \theta_{N_0}^0(t) - \theta_1^0(t) + \left( \frac{\theta_1^0(t)}{\bar{a}_1^0 + 1} + \frac{\prod_{l=1}^{N_0-1}\bar{a}_l^0}{\prod_{l=1}^{N_0-1}(\bar{a}^0_l+1)}\theta^0_{N_0}(t) + \sum_{i=2}^{N_0-1} \frac{\prod_{l=1}^{i-1} \bar{a}^0_l}{\prod_{l=1}^i(\bar{a}^0_l + 1)} \theta_i^0(t) - \theta^0_{N_0}(t)\right)\\
&+ \left(\theta^0_1(t) - \frac{\theta^0_{N_0}(t)}{\underline{a}_{N_0}^0+1} - \frac{\prod_{l=2}^{N_0}\underline{a}_l^0}{\prod_{l=2}^{N_0} (\underline{a}_l^0 +1)}\theta^0_1(t) - \sum_{i=2}^{N_0-1} \frac{\prod_{l=i+1}^{N_0}\underline{a}_l^0}{\prod_{l=i}^{N_0} (\underline{a}_l^0 +1)} \theta^0_i(t)\right)\\
& = \theta_{N_0}^0(t) - \theta_1^0(t) + \frac{1}{\bar{a}_1^0 + 1}(\theta^0_1(t) - \theta^0_{N_0}(t)) + \sum_{i=2}^{N_0-1} \frac{\prod_{l=1}^{i-1} \bar{a}^0_l}{\prod_{l=1}^i(\bar{a}^0_l + 1)}( \theta_i^0(t) - \theta^0_{N_0}(t)) \\
&+  \frac{1}{\underline{a}_{N_0}^0+1}(\theta^0_1(t) - \theta^0_{N_0}(t)) + \sum_{i=2}^{N_0-1} \frac{\prod_{l=i+1}^{N_0}\underline{a}_l^0}{\prod_{l=i}^{N_0} (\underline{a}_l^0 +1)}(\theta^0_1(t)- \theta^0_i(t)),
\end{aligned}
\end{equation}
where we apply the property that the coefficients sum of convex combination structure of $\bar{\theta}_0$ and $\underline{\theta}_0$ are respectively equal to $1$.
According to the design of coefficients \eqref{coeffi_0} and \eqref{permutation_0}, it is known that
\[\bar{a}^0_{N_0+1-i} = \underline{a}^0_i, \quad i=1,2\ldots,N_0.\]
Thus, we immediately have
\begin{equation}\label{QD_02}
\frac{\prod_{l=i+1}^{N_0}\underline{a}_l^0}{\prod_{l=i}^{N_0}(\underline{a}_l^0 +1)} = \frac{\prod_{l=1}^{N_0-i}\bar{a}_l^0}{\prod_{l=1}^{N_0+1-i}(\bar{a}_l^0 +1)}.
\end{equation}
Then we combine \eqref{QD_01} and \eqref{QD_02}to obtain that
\begin{align*}
Q^0(t) &\ge \theta_{N_0}^0(t) - \theta_1^0(t) + \frac{1}{\bar{a}_1^0 + 1}(\theta^0_1(t) - \theta^0_{N_0}(t)) +\sum_{i=2}^{N_0-1} \frac{\prod_{l=1}^{i-1} \bar{a}^0_l}{\prod\limits_{l=1}^i(\bar{a}^0_l + 1)}( \theta_1^0(t) - \theta^0_{N_0}(t))\\
&+ \frac{1}{\bar{a}_{1}^0+1}(\theta^0_1(t) - \theta^0_{N_0}(t)) + \sum_{i=2}^{N_0-1}\frac{\prod_{l=1}^{N_0-i}\bar{a}_l^0}{\prod_{l=1}^{N_0+1-i}(\bar{a}_l^0 +1)} (\theta^0_1(t) - \theta^0_{N_0}(t))\\
& = \theta_{N_0}^0(t) - \theta_1^0(t) - \left( \frac{2}{\bar{a}_1^0 + 1} + \sum_{i=2}^{N_0-1} \frac{\prod_{l=1}^{i-1} \bar{a}^0_l}{\prod_{l=1}^i(\bar{a}^0_l + 1)}  + \sum_{i=2}^{N_0-1}\frac{\prod_{l=1}^{N_0-i}\bar{a}_l^0}{\prod_{l=1}^{N_0+1-i}(\bar{a}_l^0 +1)} \right) (\theta^0_{N_0}(t)-\theta^0_1(t)),
\end{align*}
where we exploit the property of well-ordering, i.e.,
\[\theta_1^0 \le \theta_i^0 \le \theta_{N_0}^0, \quad 1 \le i \le N_0.\]
From \eqref{coeffi_0}, it is obvious that the value of coefficients $\bar{a}^0_k$'s is increasing as the subscript is decreasing, in particular,
\[\bar{a}^0_l \ge \bar{a}^0_{N_0 -1} = \eta (N_0+2), \quad 1 \le l \le N_0 -1.\]
Then for $2 \le i \le N_0-1$, we have the following estimates,
\begin{equation*}
\begin{aligned}
&\frac{1}{\bar{a}_1^0 + 1} \le \frac{1}{ \eta (N_0+2)+ 1}, \quad \frac{\prod_{l=1}^{i-1} \bar{a}^0_l}{\prod_{l=1}^i(\bar{a}^0_l + 1)} \le \frac{1}{\bar{a}_i^0 + 1} \le \frac{1}{ \eta (N_0+2)+ 1}, \\
&\frac{\prod_{l=1}^{N_0-i}\bar{a}_l^0}{\prod_{l=1}^{N_0+1-i}(\bar{a}_l^0 +1)} \le \frac{1}{\bar{a}_{N_0+1-i}^0 + 1} \le \frac{1}{ \eta (N_0+2)+ 1}.
\end{aligned}
\end{equation*}
Therefore we immediately obtain that
\begin{equation*}
Q^0(t) \ge \theta_{N_0}^0(t) - \theta_1^0(t) - \frac{2(N_0-1)}{\eta (N_0+2)+ 1}(\theta_{N_0}^0(t) - \theta_1^0(t)) = D_0 (\theta(t)) - \frac{2(N_0-1)}{\eta (N_0+2)+ 1}D_0 (\theta(t)).
\end{equation*}
Since $\frac{2(N_0-1)}{\eta (N_0+2)+ 1} \le \frac{2N_0}{\eta N_0} = \frac{2}{\eta}$, it can be obtained that
\[Q^0(t) \ge (1-\frac{2}{\eta} )D_0 (\theta(t)),\]
thus we derive the desired result.
\end{proof}

In the following, we exploit Algorithm $\mathcal{A}$ and Lemma \ref{eta_sin_inequality} to estimate the dynamics of the constructed quantity $Q^0$, i.e., the relative distance between $\bar{\mathcal{L}}_1^{N_0}(\bar{C}_{1,N_0})$ and $\underline{\mathcal{L}}_1^{N_0}(\underline{C}_{1,N_0})$, 
which will be presented in the lemma below.

\begin{lemma}\label{dynamics_Q0}
Let $\theta_i = \{\theta^0_i\}$ be the solution to system \eqref{KuM} with strong connected digraph $\mathcal{G}_0$. Moreover, for a given sufficiently small $D^\infty < \min\left\{\frac{\pi}{2},\alpha\right\}$, assume the following conditions hold,
\begin{equation}\label{condition_3}
 D_0(\theta(0)) < \alpha < \gamma < \pi, \quad \eta > \max \left\{\frac{1}{\sin \gamma}, \frac{2}{1 - \frac{\alpha}{\gamma}} \right\}, \quad \kappa > \left( 1 + \frac{\alpha}{\alpha - D(\theta(0))}\right) \frac{ \tilde{c}}{\beta D^\infty},
 \end{equation}
 where $\alpha, \gamma$ are constants and
\[\tilde{c} = \frac{D(\Omega) (\sum_{j=1}^{N_0-1}\eta^j A(2N_0,j) + 1)\gamma}{\sin \gamma}.\]
Then, the dynamics of $Q^0(t)$ is governed by the following equation
\begin{equation*}
\dot{Q}^0(t) \le D(\Omega) - \frac{\kappa}{\sum_{j=1}^{N_0-1} (\eta^jA(2N_0,j)) +1}\frac{\sin \gamma}{\gamma} Q^0(t), \quad t \in [0,+\infty),
\end{equation*}
and the phase diameter of the graph $\mathcal{G}_0$ is uniformly bounded by $\gamma$:
\[D_0(\theta(t)) < \gamma, \quad t \in [0,+\infty).\]
\end{lemma}
\begin{proof}
As the proof is rather lengthy, we put it in Appendix \ref{appendix1}.
\end{proof}

Lemma \ref{dynamics_Q0} states that the phase diameter of the digraph $\mathcal{G}_0$ is uniformly bounded and can be confined in half circle. We next show that there exists some time $t_0$ after which the phase diameter of the digraph $\mathcal{G}_0$ enters into a small region. 

\begin{lemma}\label{diameter_0_small}
Let $\theta_i = \{\theta^0_i\}$ be a solution to system \eqref{KuM}, and supose the assumptions of Lemma \ref{dynamics_Q0} hold. Then there exists time $t_0$ such that
\begin{equation*}
D_0(\theta(t)) \le D^\infty, \quad \mbox{for} \ t \in [t_0, +\infty),
\end{equation*}
where $t_0$ can be estimated as below and bounded by $\bar{t}$ given in Lemma \ref{diameter_alpha}
\begin{equation}\label{D-1}
t_0 < \frac{\alpha}{\kappa \frac{\sin\gamma}{(\sum_{j=1}^{N_0-1} \eta^j A(2N_0,j) + 1)\gamma} \beta D^\infty - D(\Omega)} < \bar{t}.
\end{equation}
\end{lemma}
\begin{proof}
In Lemma \ref{dynamics_Q0}, we have obtained that the dynamics of quantity $Q^0(t)$ is governed by the following equation
\begin{equation}\label{D-a1}
\dot{Q}^0(t) \le D(\Omega) - \kappa \frac{1}{ \sum_{j=1}^{N_0-1} (\eta^j A(2N_0,j)) + 1} \frac{\sin \gamma}{\gamma} Q^0(t), \quad t \in [0,+\infty).
\end{equation}
In the subsequence, we will find some time $t_0$ after which the quantity $Q^0$ in \eqref{D-a1} is uniformly bounded. We consider two cases separately.\newline

\noindent$\diamond$ \textbf{Case 1.} We first consider the case that $Q^0(0) > \beta D^\infty$. When $Q^0(t) \in [\beta D^\infty,Q^0(0)]$, according to \eqref{C-13}, we have
\begin{equation}
\begin{aligned}\label{D-2}
\dot{Q}^0(t) &\le D(\Omega) - \kappa \frac{\sin \gamma}{ (\sum_{j=1}^{N_0-1} (\eta^j A(2N_0,j)) + 1)\gamma} Q^0(t) \\
&\le D(\Omega) - \kappa \frac{\sin \gamma}{ (\sum_{j=1}^{N_0-1} (\eta^j A(2N_0,j)) + 1)\gamma} \beta D^\infty < 0.
\end{aligned}
\end{equation}
This means that when $Q^0(t)$ is located in the interval $[\beta D^\infty,Q^0(0)]$, $Q^0(t)$ will keep decreasing with a uniform rate. Therefore we can define a stopping time $t_0$ as follows,
\begin{equation*}
t_0 = \inf \{t \ge 0 \ | \ Q^0(t) \le \beta D^\infty\}.
\end{equation*}
Then, according to the definition of $t_0$, we know that $Q^0$ will decrease before $t_0$ and has the following property at $t_0$,
\begin{equation}\label{D-c2}
Q^0(t_0) = \beta D^\infty.
\end{equation}
Moerover, according to \eqref{D-2}, it is obvious that the stopping time $t_0$ satisfies the following upper bound estimate,
\begin{equation}\label{D-b2}
t_0 \le \frac{Q^0(0) - \beta D^\infty }{\kappa \frac{\sin\gamma}{(\sum_{j=1}^{N_0-1} \eta^j A(2N_0,j) + 1)\gamma} \beta D^\infty - D(\Omega)}.
\end{equation}
Now we study the upper bound of $Q^0$ on $[t_0, +\infty)$. Coming back to \eqref{D-2}, we can apply \eqref{D-c2} and the same argument in \eqref{C-14} to derive
\begin{equation}\label{D-d2}
Q^0(t) \le \beta D^\infty, \ t \in [t_0,+\infty).
\end{equation}

\noindent $\diamond$ \textbf{Case 2.} For another case that $Q^0(0) \le \beta D^\infty$. Applying the same analysis in \eqref{C-14}, we get
\begin{equation}\label{D-a3}
Q^0(t) \le \beta D^\infty, \quad t \in [0,+\infty).
\end{equation}
This allows us to directly set $t_0=0$. \newline

Thus we apply \eqref{D-d2}, \eqref{D-a3}, and Lemma \ref{beta_QD} to estimate the upper bound of $D_0(\theta)$ on $[t_0, \infty)$ as below
 \begin{equation}\label{D-3}
D_0(\theta(t)) \le \frac{Q^0(t)}{\beta} \le D^\infty, \quad \mbox{for} \ t \in [t_0, +\infty).
\end{equation}

On the other hand, in order to verify \eqref{D-1}, we do further estimates on $t_0$ in \eqref{D-3}. It is known from \eqref{D-b2} in Case 1 and $t_0 = 0$ in Case 2 that
\begin{equation}\label{D-4}
t_0 < \frac{\alpha}{\kappa \frac{\sin\gamma}{(\sum_{j=1}^{N_0-1} \eta^j A(2N_0,j) + 1)\gamma} \beta D^\infty - D(\Omega)}.
\end{equation}
Here, we use the truth that $Q^0(0) < \alpha$. Thus, from the assumption of $\kappa$ in \eqref{condition_3}, i.e.,
\[\kappa > \left( 1 + \frac{\alpha}{\alpha - D(\theta(0))}\right) \frac{\tilde{c}}{\beta D^\infty}, \quad \tilde{c} = \frac{D(\Omega) (\sum_{j=1}^{N_0-1}\eta^j A(2N_0,j) + 1)\gamma}{\sin \gamma}.\]
it yields that the time $t_0$ has the following estimate,
\begin{equation}\label{D-5}
t_0 < \frac{\alpha}{(1 + \frac{\alpha}{\alpha - D(\theta(0))})D(\Omega) - D(\Omega)} = \frac{\alpha - D(\theta(0))}{D(\Omega)} = \bar{t}.
\end{equation}
Thus, we derive the desired results \eqref{D-3}, \eqref{D-4} and \eqref{D-5}.
\end{proof}

\section{General network}\label{sec:4}
\setcounter{equation}{0}
Now, we focus on the general network, and provide a proof of Theorem \ref{enter_small} for the emergence of complete synchronization in Kuramoto model with general network containing a spanning tree. According to Definition \ref{node} and Lemma \ref{one_maximum}, the digraph $\mathcal{G}$ associated to system \eqref{KuM} has a unique maximum node if it contains a spanning tree structure. From Remark \ref{w.l.o.g.}, without loss of generality, $\mathcal{G}$ can be decomposed into a union as $\mathcal{G} = \bigcup_{i=0}^d \mathcal{G}_i$, where $\mathcal{G}_p$ is a maximum node of $\mathcal{G}\setminus (\bigcup_{i=0}^{p-1} \mathcal{G}_i)$.

In Section \ref{sec:3}, for the situation that $d=0$, we showed that the phase diameter of the digraph $\mathcal{G}_0$ is uniformly bounded and can be confined in a quarter circle after some finite time. However, for the case that $d > 0$, $\mathcal{G}_k$'s are not maximum nodes in $\mathcal{G}$ for $k \ge 1$. Hence, we can not directly apply the same method in Lemma \ref{dynamics_Q0} and Lemma \ref{diameter_0_small} for the situation $d=0$. More precisely, the oscillators in $\mathcal{G}_i$ with $i < k$ perform as an attraction source and influence the agents in $\mathcal{G}_k$. Thus we can not ignore the information from $\mathcal{G}_i$ with $i < k$ when we study the behavior of agents in $\mathcal{G}_k$.

From Remark \ref{w.l.o.g.} and node decomposition, the graph $\mathcal{G}$ can be represented as 
\begin{equation*}
\mathcal{G} = \bigcup_{k=0}^d \mathcal{G}_k, \quad |\mathcal{G}_k| = N_k,
\end{equation*}
and we denote the oscillators in $\mathcal{G}_k$ by $\theta^k_i$ with $1 \le i \le N_k$. Then we assume that at time $t$, the oscillators in each $\mathcal{G}_k$ are well-ordered as below:
\begin{equation}\label{well_order_k}
\theta^k_1(t) \le \theta^k_2(t) \le \ldots \le \theta^k_{N_k}(t),  \quad 0 \le k \le d.
\end{equation}
For each subdigraph $\mathcal{G}_k$ with $k \ge 0$ which is strongly connected, we follow the process in  Algorithm $\mathcal{A}_1$ and $\mathcal{A}_2$ to construct $\bar{\mathcal{L}}_{l-1}^{N_k}(\bar{C}_{l-1,N_k})$ and $\underline{\mathcal{L}}_1^{l+1}(\underline{C}_{1,l+1})$ by redesigning the coefficients $\bar{a}^k_l$ and $\underline{a}_l^k$ of convex combination as below:

\begin{equation}\label{coeffi_k}
\begin{cases}
\displaystyle \bar{\mathcal{L}}_{l-1}^{N_k}(\bar{C}_{l-1,N_k}) \ \mbox{with} \ \bar{a}^k_{N_k} = 0, \ \bar{a}^k_{l-1} =\eta (2N -l +2)(\bar{a}^k_l + 1), \quad 2 \le l \le N_k, \\
\displaystyle \underline{\mathcal{L}}_1^{l+1}(\underline{C}_{1,l+1}) \ \mbox{with} \ \underline{a}^k_1 = 0, \ \underline{a}^k_{l+1} = \eta(l+1+2N-N_k)(\underline{a}^k_l + 1), \quad 1\le l \le N_k-1,
\end{cases}
\end{equation}
By induction principle, we deduce that
\begin{equation}\label{permutation_k}
\begin{cases}
\displaystyle \bar{a}^k_{l-1} = \sum_{j=1}^{N_k - l+1} \eta^j A(2N-l+2,j), \quad 2 \le l \le N_k, \\
\displaystyle\underline{a}^k_{l+1} = \sum_{j=1}^l \eta^j A(l+1+ 2N-N_k,j), \quad 1 \le l \le N_k-1.
\end{cases}
\end{equation}
Note that $\bar{a}^k_{N_k+1-i} = \underline{a}^k_i, \ i=1,2\ldots,N_k$. By simple calculation, we have
\begin{equation}\label{a^k_1-size}
\bar{a}^k_1 = \sum^{N_k -1}_{j=1} (\eta^j A(2N,j)),  \quad \bar{a}^k_1 \le \sum^{N -1}_{j=1} (\eta^j A(2N,j)), \quad 0 \le k\le d.
\end{equation}
And we set the following notations, 
\begin{align}
&\bar{\theta}^k_l := \bar{\mathcal{L}}_l^{N_k}(\bar{C}_{l,N_k}),\quad \underline{\theta}^k_l := \underline{\mathcal{L}}_1^l(\underline{C}_{1,l}), \quad 1 \le l \le N_k, \quad 0 \le k \le d,\label{abbreviation_k}\\
&\bar{\theta}_k := \bar{\mathcal{L}}_1^{N_k}(\bar{C}_{1,N_k}), \quad \underline{\theta}_k := \underline{\mathcal{L}}_1^{N_k}(\underline{C}_{1,N_k}), \quad 0\le k \le d,\label{bar_underline_k}\\
&Q^k(t) := \max_{0 \le i\le k}\{\bar{\theta}_i\} - \min_{0\le i\le k}\{\underline{\theta}_i\},  \quad 0 \le k \le d.\label{definition_Qk}
\end{align}
Due to the analyticity of the solution, $Q^k(t)$ is Lipschitz continuous. Similar as in Section \ref{sec:3}, in the following, we will first establish the comparison relation between the quantity $Q^k(t)$ and phase diameter $D_k(\theta(t))$ of the first $k+1$ nodes, which plays an important role in the later analysis.

\begin{lemma}\label{beta_QD_k}
Let $\theta_i$ be a solution to system \eqref{KuM} and assume that for each subdigraph $\mathcal{G}_k$, the coefficients $\bar{a}^k_l$ and $\underline{a}^k_l$ of convex combination in Algorithm $\mathcal{A}$ satisfy the scheme \eqref{coeffi_k}. Then at each time $t$, we have the following relation
\begin{equation*}
\beta D_k(\theta(t)) \le Q^k(t) \le D_k(\theta(t)),  \quad 0\le k \le d, \quad\beta = 1 - \frac{2}{\eta},
\end{equation*}
where $D_k(\theta) = \max\limits_{0 \le i \le k} \max\limits_{1 \le j \le N_i} \{\theta_j^i\} - \min\limits_{0 \le i \le k} \min\limits_{1 \le j \le N_i} \{\theta_j^i\}$ and $\eta$ satisfies the condition \eqref{condition_1}. 
\end{lemma}
\begin{proof}
Without loss of generality, assume that at time $t$, the oscillators in each subdigraph $\mathcal{G}_k$ are all well-ordered as below
\begin{equation}\label{E-1}
\theta^k_1 \le \theta^k_2 \le \ldots \le \theta^k_{N_k},  \quad 0 \le k \le d.
\end{equation}
From the definition of the quantity $Q^k(t)$ in \eqref{definition_Qk} and the convex combination structure of $\bar{\theta}_k$ and $\underline{\theta}_k$ in \eqref{bar_underline_k}, it can be directly derived that
\begin{equation}\label{E-2}
\bar{\theta}_k \le \theta^k_{N_k}, \quad \underline{\theta}_k \ge \theta^k_1, \quad Q^k(t)=\max_{0 \le i\le k}\{\bar{\theta}_i\} - \min_{0\le i\le k}\{\underline{\theta}_i\} \le \max_{0 \le i \le k}\{\theta^i_{N_i}\} - \min_{0 \le i \le k}\{\theta^i_{1}\}.
\end{equation}
This means that
\begin{equation*}
Q^k(t) \le D_k(\theta(t)) =  \max_{0 \le i \le k} \max\limits_{1 \le j \le N_i} \{\theta_j^i\} - \min_{0 \le i \le k} \min\limits_{1 \le j \le N_i} \{\theta_j^i\}.
\end{equation*}
Next we will prove the left part of this Lemma. In fact, we denote the extreme phases of the first $k+1$ nodes by
\begin{equation}\label{E-3}
\theta^p_{N_p} := \max_{0 \le i \le k}\max_{1 \le j\le N_i} \{\theta^i_j\}, \quad \theta^q_1 := \min_{0\le i\le k}\min_{1 \le j \le N_i} \{\theta^i_j\}, \quad 0\le p \le k, \ 0\le q \le k.
\end{equation}
It is clear that $D_k(\theta(t)) = \theta^p_{N_p} - \theta^q_1$.
We consider two cases separately.\newline

\noindent$\bullet$ \textbf{Case 1.} If the index satisfy the relation $p = q$, we have
\begin{equation*}
Q^k(t) = \max_{0 \le i\le k}\{\bar{\theta}_i\} - \min_{0\le i\le k}\{\underline{\theta}_i\} \ge \bar{\theta}_p  - \underline{\theta}_p = \theta^p_{N_p} - \theta^p_1 + \bar{\theta}_p - \theta^p_{N_p}  + \theta^p_1 - \underline{\theta}_p.
\end{equation*}
In this case, applying the same arguments in Lemma \ref{beta_QD}, we obtain that
\begin{equation*}
\begin{aligned}
Q^k(t) &\ge \theta^p_{N_p}(t) - \theta^p_1(t) - \frac{2(N_p-1)}{\eta (N+2)+ 1}(\theta_{N_p}^p(t) - \theta_1^p(t)) = D_k (\theta(t)) - \frac{2(N_p-1)}{\eta (N+2)+ 1}D_k (\theta(t))\\
&\ge (1-\frac{2}{\eta})D_k (\theta(t)) = \beta D_k (\theta(t)).
\end{aligned}
\end{equation*}
Here, in the above estimates, based on the construction of coefficients of convex combination in \eqref{coeffi_k}, we used the inequalities
\begin{equation}\label{E-4}
\bar{a}^k_l \ge \bar{a}^k_{N_k-1} = \eta (2N - N_k + 2) \ge \eta(N +2), \quad 0\le k \le d, \ 1 \le l \le N_k -1,
\end{equation}
and applied the symmetric property
\begin{equation}\label{E-5}
\bar{a}^k_{N_k + 1-l} = \underline{a}^k_l, \quad 1 \le l \le N_k, \ 0\le k \le d.
\end{equation}

\noindent $\bullet$ \textbf{Case 2.} Consider the case that $p\ne q$, then we have
\begin{equation}\label{E-6}
\begin{aligned}
Q^k(t) &=\max_{0 \le i\le k}\{\bar{\theta}_i\} - \min_{0\le i\le k}\{\underline{\theta}_i\}  \ge \bar{\theta}_p - \underline{\theta}_q = \bar{\theta}_p(t) - \theta_{N_p}^p(t) + \theta_{N_p}^p(t) - \theta_1^q(t) + \theta^q_1(t) - \underline{\theta}_q(t) \\
&= \theta_{N_p}^p(t) - \theta_1^q(t) + \bar{\theta}_p(t) - \theta_{N_p}^p(t) + \theta^q_1(t) - \underline{\theta}_q(t) \\
& = \theta_{N_p}^p(t) - \theta_1^q(t) + \left( \frac{\theta_1^p(t)}{\bar{a}_1^p + 1} + \frac{\prod_{l=1}^{N_p-1}\bar{a}_l^p}{\prod_{l=1}^{N_p-1}(\bar{a}^p_l+1)}\theta^p_{N_p}(t) + \sum_{i=2}^{N_p-1} \frac{\prod_{l=1}^{i-1} \bar{a}^p_l}{\prod_{l=1}^i(\bar{a}^p_l + 1)} \theta_i^p(t) - \theta^p_{N_p}(t)\right)\\
&+ \left(\theta^q_1(t) - \frac{\theta^q_{N_q}(t)}{\underline{a}_{N_q}^q+1} - \frac{\prod_{l=2}^{N_q}\underline{a}_l^q}{\prod_{l=2}^{N_q} (\underline{a}_l^q +1)}\theta^q_1(t) - \sum_{i=2}^{N_q-1} \frac{\prod_{l=i+1}^{N_q}\underline{a}_l^q}{\prod_{l=i}^{N_q} (\underline{a}_l^q +1)} \theta^q_i(t)\right)\\
& = \theta_{N_p}^p(t) - \theta_1^q(t) + \frac{1}{\bar{a}_1^p + 1}(\theta^p_1(t) - \theta^p_{N_p}(t)) + \sum_{i=2}^{N_p-1} \frac{\prod_{l=1}^{i-1} \bar{a}^p_l}{\prod_{l=1}^i(\bar{a}^p_l + 1)}( \theta_i^p(t) - \theta^p_{N_p}(t)) \\
&+  \frac{1}{\underline{a}_{N_q}^q+1}(\theta^q_1(t) - \theta^q_{N_q}(t)) + \sum_{i=2}^{N_q-1} \frac{\prod_{l=i+1}^{N_q}\underline{a}_l^q}{\prod_{l=i}^{N_q} (\underline{a}_l^q +1)}(\theta^q_1(t)- \theta^q_i(t)),
\end{aligned}
\end{equation}
where we apply the property that the coefficients sum of convex combination of $\bar{\theta}_k$ and $\underline{\theta}_k$ with $0\le k \le d$ are respectively equal to $1$. Moreover, we know from \eqref{E-3} that
\begin{equation*}
 \theta^p_i \ge \theta^q_1, \ i=1,2,\ldots,N_p -1, \qquad \theta^q_{i} \le \theta^p_{N_p}, \ i = 2,\ldots,N_q-1, N_q.
\end{equation*} 
This implies that
\begin{equation}\label{E-7}
\theta^p_i - \theta^p_{N_p} \ge \theta^q_1 - \theta^p_{N_p}, \ i=1,2,\ldots,N_p -1, \quad \theta^q_1 - \theta^q_{i} \ge \theta^q_1 - \theta^p_{N_p}, \ i = 2,\ldots,N_q-1, N_q.
\end{equation}
Moreover, exlpoiting the symmetric property \eqref{E-5}, we immediately have
\begin{equation}\label{E-8}
\frac{\prod_{l=i+1}^{N_q}\underline{a}_l^q}{\prod_{l=i}^{N_q}(\underline{a}_l^q +1)} = \frac{\prod_{l=1}^{N_q-i}\bar{a}_l^q}{\prod_{l=1}^{N_q+1-i}(\bar{a}_l^q +1)}.
\end{equation}
Therefore, combining \eqref{E-6}, \eqref{E-7} and \eqref{E-8}, we obtain that
\begin{equation*}
\begin{aligned}
Q^k(t) &\ge \theta_{N_p}^p(t) - \theta_1^q(t) + \frac{1}{\bar{a}_1^p + 1}(\theta^q_1(t) - \theta^p_{N_p}(t)) + \sum_{i=2}^{N_p-1} \frac{\prod_{l=1}^{i-1} \bar{a}^p_l}{\prod_{l=1}^i(\bar{a}^p_l + 1)}( \theta_1^q(t) - \theta^p_{N_p}(t)) \\
&+  \frac{1}{\bar{a}_{1}^q+1}(\theta^q_1(t) - \theta^p_{N_p}(t)) + \sum_{i=2}^{N_q-1} \frac{\prod_{l=1}^{N_q-i}\bar{a}_l^q}{\prod_{l=1}^{N_q+1-i}(\bar{a}_l^q +1)}(\theta^q_1(t)- \theta^p_{N_p}(t)) \\
& = \theta_{N_p}^p(t) - \theta_1^q(t) \\
&- \left( \underbrace{\frac{1}{\bar{a}_1^p + 1} + \sum_{i=2}^{N_p-1} \frac{\prod_{l=1}^{i-1} \bar{a}^p_l}{\prod_{l=1}^i(\bar{a}^p_l + 1)} + \frac{1}{\bar{a}_{1}^q+1} + \sum_{i=2}^{N_q-1} \frac{\prod_{l=1}^{N_q-i}\bar{a}_l^q}{\prod_{l=1}^{N_q+1-i}(\bar{a}_l^q +1)}}\right)(\theta_{N_p}^p(t) - \theta_1^q(t)).
\end{aligned}
\end{equation*}
We apply \eqref{E-4} and estimate the items in the above brace respectively,
\begin{equation*}
\begin{aligned}
&\frac{1}{\bar{a}_1^p + 1} \le \frac{1}{ \eta (N+2)+ 1}, \quad \frac{1}{\bar{a}_1^q + 1} \le \frac{1}{ \eta (N+2)+ 1}, \\
&\frac{\prod_{l=1}^{i-1} \bar{a}^p_l}{\prod_{l=1}^i(\bar{a}^p_l + 1)} \le \frac{1}{\bar{a}_i^p + 1} \le \frac{1}{ \eta (N+2)+ 1}, \quad 2 \le i \le N_p-1,\\
&\frac{\prod_{l=1}^{N_q-i}\bar{a}_l^q}{\prod_{l=1}^{N_q+1-i}(\bar{a}_l^q +1)} \le \frac{1}{\bar{a}_{N_q+1-i}^q + 1} \le \frac{1}{ \eta (N+2)+ 1}, \quad 2 \le i \le N_q-1.
\end{aligned}
\end{equation*}
Thus, based on the above estimates, we have
\begin{equation*}
Q^k(t) \ge \theta_{N_p}^p(t) - \theta_1^q(t) - \frac{N_p + N_q -2}{\eta (N+2)+ 1}(\theta_{N_p}^p(t) - \theta_1^q(t)) = D_k (\theta(t)) - \frac{N_p + N_q -2}{\eta (N+2)+ 1}D_k (\theta(t).
\end{equation*}
Since $\frac{N_p + N_q -2}{\eta (N+2)+ 1} \le \frac{2N}{\eta N} = \frac{2}{\eta}$ and from \eqref{E-3}, we immediately have
\[Q^k(t) \ge (1-\frac{2}{\eta} )D_k (\theta(t)) = \beta D_k (\theta(t)).\]
Thus combining the above analysis, we derive the desired result.
\end{proof}

Now, we are ready to prove the main Theorem \ref{enter_small}. We will follow similar arguments as in Section \ref{sec:3} to finish the proof. Actually, we will study the constructed quantity $Q^k(t)$ which contains the information from $\mathcal{G}_i$ with $i <k$, and then yield the hypo-coercivity of the diameter. Following similar arguments in Lemma \ref{dynamics_Q0} and Lemma \ref{diameter_0_small}, we have the following estimates for the first maximal node $\mathcal{G}_0$.

\begin{lemma}\label{L4.2}
Suppose that the network topology contains a spanning tree, and let $\theta_i$ be a solution to \eqref{KuM}. Moreover, assume that the initial data and the quantity $\eta$ satisfy
 \begin{equation}\label{condition_11}
 D(\theta(0)) < \alpha < \gamma < \pi, \quad \eta > \max \left\{\frac{1}{\sin \gamma}, \frac{2}{1 - \frac{\alpha}{\gamma}} \right\},
 \end{equation}
 where $\alpha, \gamma$ are positive constants. For a given sufficiently small $D^\infty < \min\{\frac{\pi}{2},\alpha\}$, if the coupling strength $\kappa$ satisfies
 \begin{equation}\label{condition_21}
 \kappa > \left( 1 + \frac{(d+1)\alpha}{\alpha - D(\theta(0))}\right) \frac{(4c)^d \tilde{c}}{\beta^{d+1}D^\infty},
 \end{equation}
 where $d$ is the number of general nodes and
\[c = \frac{(2N + 1)(\sum_{j=1}^{N-1} \eta^j A(2N,j) + 1)\gamma}{\sin \gamma}, \quad \tilde{c} = \frac{D(\Omega) (\sum_{j=1}^{N-1}\eta^j A(2N,j) + 1)\gamma}{\sin \gamma},\]
 then the following two assertions hold for the maximum node $\mathcal{G}_0$:
 \begin{enumerate}
\item  The dynamics of $Q^0(t)$ is governed by the following equation
\begin{equation*}
\dot{Q}^0(t) \le D(\Omega) - \frac{\kappa}{\sum_{j=1}^{N-1} (\eta^jA(2N,j)) +1}\frac{\sin \gamma}{\gamma} Q^0(t), \quad t \in [0,+\infty),
\end{equation*}
\item  there exists time $t_0$ such that
\begin{equation*}
D_0(\theta(t)) \le \frac{\beta^d D^\infty}{(4c)^d}, \quad \mbox{for} \ t \in [t_0, +\infty),
\end{equation*}
where $t_0$ can be estimated as below and bounded by $\bar{t}$ given in Lemma \ref{diameter_alpha}
\begin{equation*}
t_0 < \frac{\alpha}{\kappa \frac{\sin\gamma}{(\sum_{j=1}^{N-1} \eta^j A(2N,j) + 1)\gamma} \frac{\beta^{d+1}D^\infty}{(4c)^d} - D(\Omega)} < \bar{t}.
\end{equation*}

\end{enumerate}

\end{lemma}
Since the proof is almost the same as that in Lemma \ref{dynamics_Q0} and Lemma \ref{diameter_0_small}, we omit its details. Inspiring from Lemma \ref{L4.2}, we make the following reasonable ansatz for $Q^k(t)$ for $0 \le k \le d$.\newline 

\noindent \textbf{Ansatz}:  
\begin{enumerate}
\item  The dynamics of $Q^{k}(t)$ is governed by the following differential inequality,
\begin{equation}\label{F-1}
\dot{Q}^{k}(t) \le D(\Omega) - \frac{\kappa}{\sum_{j=1}^{N-1} (\eta^jA(2N,j)) +1}\frac{\sin \gamma}{\gamma} Q^{k}(t) + \kappa(2N+1)D_{k-1}(\theta(t)), \quad t \in [0,+\infty),
\end{equation}
where we assume $D_{-1}(\theta(t)) = 0$.
\item There exists a finite time $t_k$ such that, the phase diameter $D_k(\theta)$ of $\bigcup_{i=0}^k\mathcal{G}_i$ is uniformly bounded after $t_k$, i.e.,
\begin{equation}\label{F-2}
D_k(\theta(t)) \le \frac{\beta^{d-k} D^\infty}{(4c)^{d-k}}, \quad \mbox{for} \ t \in [t_k, +\infty),
\end{equation}
where $t_k$ can be estimated as below
\begin{equation}\label{F-3}
t_k < \frac{(k+1)\alpha}{\kappa \frac{\sin\gamma}{(\sum_{j=1}^{N-1} \eta^j A(2N,j) + 1)\gamma} \frac{\beta^{d+1}D^\infty}{(4c)^d} - D(\Omega)} < \bar{t} = \frac{\alpha - D(\theta(0))}{D(\Omega)}.
\end{equation}
\newline
\end{enumerate}

In the following, we will verify the ansatz respectively in two lemmas by induction criteria. More precisely, suppose the ansatz holds for $ Q^k$ and $D_k(\theta)$ with $0 \le k \le d-1$, we will prove that the ansatz also holds for $Q^{k+1}$ and $D_{k+1}(\theta)$.

\begin{lemma}\label{L4.3}
Suppose the conditions in Lemma \ref{L4.2} are fulfilled, and the ansatz in \eqref{F-1}, \eqref{F-2} and \eqref{F-3} holds for some $k$ with $0 \le k \le d-1$. Then the ansatz  \eqref{F-1} holds for $k+1$.
\end{lemma}
\begin{proof}
Similar as before, we will use proof by contradiction criteria to verify the ansatz for $Q^{k+1}$. To this end, we first define a set below,
\begin{equation*}
\mathcal{B}_{k+1} = \{T>0 \  : \ D_{k+1}(\theta(t)) < \gamma, \ \forall \ t \in [0,T)\}.
\end{equation*}
From Lemma \ref{diameter_alpha}, we know that
\begin{equation*}
D_{k+1}(\theta(t)) \le D(\theta(t)) < \alpha < \gamma, \quad \forall \ t \in [0, \bar{t}).
\end{equation*}
It is clear that $\bar{t} \in \mathcal{B}_{k+1}$. Thus the set $\mathcal{B}_{k+1}$ is not empty. We define $T^* = \sup \mathcal{B}_{k+1}$, and will prove by contradiction that $T^* = +\infty$. Suppose not, i.e., $T^* < +\infty$. It is obvious that
\begin{equation}\label{F-4}
\bar{t} \le T^*, \quad D_{k+1}(\theta(t)) < \gamma, \ \forall \ t \in [0,T^*), \quad D_{k+1}(\theta(T^*)) = \gamma.
\end{equation}

Since the solution to system \eqref{KuM} is analytic, in the finite time interval $[0,T^*)$, $\bar{\theta}_i$ and $\bar{\theta}_j$ either collide finite times or always stay together. Similar to the analysis in Lemma \ref{dynamics_Q0}, without loss of generality, we only consider the situation that there is no pair of $\bar{\theta}_i$ and $\bar{\theta}_j$  staying together through all period $[0,T^*)$. That means the order of $\{\bar{\theta}_i\}_{i=0}^{k+1}$ will only exchange finite times in $[0,T^*)$, so does $\{\underline{\theta}_i\}_{i=0}^{k+1}$ . Thus, we divide the time interval $[0,T^*)$ into a finite union as below
\[[0,T^*) = \bigcup_{l=1}^r J_l, \quad J_l = [t_{l-1},t_l).\]
such that the orders of both $\{\bar{\theta}_i\}_{i=0}^{k+1}$ and $\{\underline{\theta}_i\}_{i=0}^{k+1}$ are preseved in each interval $J_l$. In the following, we will show the contradiction in two steps.\newline

\noindent $\star$ \textbf{Step 1.} In this step, we first verify the Claim \eqref{F-1} holds for $Q^{k+1}$ on $[0, T^*)$, i.e.,
\begin{equation}\label{F-a20}
\dot{Q}^{k+1}(t) \le D(\Omega) - \kappa \frac{1}{\sum_{j=1}^{N-1} (\eta^j A(2N,j)) + 1} \frac{\sin \gamma}{\gamma} Q^{k+1}(t)+ \kappa (2N+1) D_k(\theta(t)), \quad \mbox{on} \ [0, T^*).
\end{equation}
As the proof is rather lengthy, we put the detailed proof in Appendix \ref{appendix 2}. 
\newline

\noindent $\star$ \textbf{Step 2.} In this step, we will study the upper bound of $Q^{k+1}$ in \eqref{F-a20} in time interval $[t_k,T^*)$, where $t_k$ is defined in Ansatz $\eqref{F-2}$ for $D_k(\theta)$. For the sake of discussion, we rewrite the equation \eqref{F-a20}
\begin{equation}\label{F-d20}
\dot{Q}^{k+1}(t) \le -\kappa \frac{\sin \gamma}{(\sum_{j=1}^{N-1} (\eta^j A(2N,j)) + 1)\gamma} \left(Q^{k+1}(t) - cD_k(\theta(t)) -\frac{\tilde{c}}{\kappa}\right), \quad t\in [0,T^*),
\end{equation}
where the expressions of $c$ and $\tilde{c}$ are given as below
\begin{equation}\label{F-dd20}
c = \frac{(2N+1)(\sum_{j=1}^{N-1} (\eta^j A(2N,j)) + 1)\gamma }{\sin \gamma}\quad \mbox{and} \quad \tilde{c} = \frac{D(\Omega)(\sum_{j=1}^{N-1} (\eta^j A(2N,j)) + 1)\gamma }{\sin \gamma}.
\end{equation}
For the term $D_k(\theta)$ in \eqref{F-d20}, 
by induction principle, we have assumed that the Claim \eqref{F-2} holds for $D_k(\theta)$, i.e., there exists time $t_k$ such that
\begin{equation}\label{F-21}
 D_k(\theta(t)) \le \frac{\beta^{d-k} D^\infty}{(4c)^{d-k}}, \quad \ t \in [t_k, +\infty), \quad t_k < \bar{t}.
\end{equation}
For the term $\frac{\tilde{c}}{\kappa}$ in \eqref{F-d20}, from the condition \eqref{condition_2}, it is obvious that
\begin{equation*}
\kappa > \left( 1 + \frac{(d+1)\alpha}{\alpha - D(\theta(0))}\right) \frac{(4c)^d \tilde{c}}{\beta^{d+1}D^\infty} > \frac{(4c)^d \tilde{c}}{\beta^{d+1}D^\infty},
\end{equation*}
which directly yields that
\begin{equation}\label{F-22}
\frac{\tilde{c}}{\kappa} < \frac{\beta^{d+1}D^\infty}{(4c)^d} < \frac{\beta^{d-k}D^\infty}{4^{d-k}c^{d-k-1}}, \quad \mbox{where} \  0 \le k \le d-1, \quad \beta < 1, \quad c > 1.
\end{equation}
 Then we add the esimates of the two terms $D_k(\theta)$ and $\frac{\tilde{c}}{\kappa}$ in \eqref{F-21} and \eqref{F-22} to get
\begin{equation}\label{F-23}
cD_k(\theta(t)) +\frac{\tilde{c}}{\kappa} < c \frac{\beta^{d-k} D^\infty}{(4c)^{d-k}} + \frac{\beta^{d-k}D^\infty}{4^{d-k}c^{d-k-1}} < \frac{\beta^{d-k}D^\infty}{ 2(4c)^{d-k-1}} < \frac{\beta^{d-k}D^\infty}{ (4c)^{d-k-1}} , \quad t \in [t_k, +\infty).
\end{equation}
Since $t_k < \bar{t} \le T^*$ where $\bar{t}$ is obtained in Lemma \ref{diameter_alpha}, it makes sense when we consider the time interval $[t_k, T^*)$. Now based on the above estiamte \eqref{F-23}, we apply the differential equation \eqref{F-d20} and study the upper bound of $Q^{k+1}$ on $[t_k, T^*)$. We claim that
\begin{equation}\label{F-a23}
Q^{k+1}(t) \le \max \left\{Q^{k+1}(t_k), \frac{\beta^{d-k}D^\infty}{ (4c)^{d-k-1}} \right\} := M_{k+1}, \quad t \in [t_k, T^*).
\end{equation}
Suppose not, then there exists some $\tilde{t} \in (t_k, T^*)$ such that $Q^{k+1}(\tilde{t}) > M_{k+1}$. We construct a set
\[\mathcal{C}_{k+1} := \{t_k \le t < \tilde{t} : Q^{k+1}(t) \le M_{k+1}\}.\]
Since $Q^{k+1}(t_k) \le M_{k+1}$, the set $\mathcal{C}_{k+1}$ is not empty. Then we denote $t^* = \sup \mathcal{C}_{k+1}$. It is easy to see that
\begin{equation}\label{F-b23}
t^* < \tilde{t}, \quad Q^{k+1}(t^*)=M_{k+1}, \quad Q^{k+1}(t) > M_{k+1} \quad \mbox{for} \ t \in (t^*, \tilde{t}].
\end{equation}
According to the construction of $M_{k+1}$, \eqref{F-23} and \eqref{F-b23}, it is clear that for $t \in (t^*, \tilde{t}]$
\begin{equation*}
\begin{aligned}
&-\kappa \frac{\sin \gamma}{(\sum_{j=1}^{N-1} (\eta^j A(2N,j)) + 1)\gamma} \left(Q^{k+1}(t) - cD_k(\theta(t)) -\frac{\tilde{c}}{\kappa}\right)\\
& < -\kappa \frac{\sin \gamma}{(\sum_{j=1}^{N-1} (\eta^j A(2N,j)) + 1)\gamma} \left(M_{k+1} - \frac{\beta^{d-k}D^\infty}{ (4c)^{d-k-1}}\right) \le 0.
\end{aligned}
\end{equation*}
Apply the above inequality and integrate on both sides of $\eqref{F-d20}$ from $t^*$ to $\tilde{t}$ to get
\begin{equation*}
\begin{aligned}
0 < Q^{k+1}(\tilde{t}) - M_{k+1} &= Q^{k+1}(\tilde{t}) - Q^{k+1}(t^*) \\
&\le \int_{t^*}^{\tilde{t}} -\kappa \frac{\sin \gamma}{(\sum_{j=1}^{N-1} (\eta^j A(2N,j)) + 1)\gamma} \left(Q^{k+1}(t) - cD_k(\theta(t)) -\frac{\tilde{c}}{\kappa}\right) dt < 0,
\end{aligned}
\end{equation*}
which is an obvious contradiction. Thus we complete the proof of \eqref{F-a23}.\newline

\noindent $\star$ \textbf{Step 3.} In this step, we will construct a contradiction to \eqref{F-4}.
According to \eqref{F-a23}, Lemma \ref{diameter_alpha} and the fact that
\begin{equation*}
\frac{\beta^{d-k}D^\infty}{ (4c)^{d-k-1}} < D^\infty, \quad t_k < \bar{t}, \quad Q^{k+1}(t_k) \le D_{k+1}(\theta(t_k)) \le D(\theta(t_k)) < \alpha,
\end{equation*}
it yields that
\begin{equation*}
Q^{k+1}(t) \le \max \left\{Q^{k+1}(t_k), \frac{\beta^{d-k}D^\infty}{ (4c)^{d-k-1}} \right\} < \max \{\alpha, D^\infty\} = \alpha, \quad t \in [t_k, T^*).
\end{equation*}
Applying Lemma \ref{beta_QD_k} and the condition \eqref{condition_1}, we immediately have
\begin{equation*}
D_{k+1}(\theta(t)) \le \frac{Q^{k+1}(t)}{\beta} < \frac{\alpha}{\beta} < \gamma, \quad t \in [t_k, T^*).
\end{equation*}
Due to the continuity of $D_{k+1}(\theta(t))$, we have
\begin{equation*}
D_{k+1}(\theta(T^*)) = \lim_{t \to (T^*)^-} D_{k+1}(\theta(t)) \le \frac{\alpha}{\beta} < \gamma,
\end{equation*}
which obviously contradicts to the assumption $D_{k+1}(\theta(T^*)) = \gamma$ in \eqref{F-4}. \newline

Thus, we combine all above analysis to conclude that $T^* = +\infty$, that is to say,
\begin{equation}\label{F-c23}
D_{k+1}(\theta(t)) < \gamma, \quad \forall \ t \in [0, +\infty).
\end{equation}
Then for any finite time $T>0$, according to \eqref{F-c23}, we can repeat the analysis in Step 1 to obtain that the differential inequality \eqref{F-1} holds for $Q^{k+1}$ on $[0,T)$. This yields the dynamics of $Q^{k+1}$ in whole time interval as below:
\begin{equation}\label{F-24}
\dot{Q}^{k+1}(t) \le D(\Omega) - \kappa \frac{1}{\sum_{j=1}^{N-1} (\eta^j A(2N,j)) + 1} \frac{\sin \gamma}{\gamma} Q^{k+1}(t)+ \kappa (2N+1) D_k(\theta(t)), \quad \mbox{on} \ [0, +\infty).
\end{equation}
Therefore, we complete the proof of the Claim \eqref{F-1} for $Q^{k+1}$.

\end{proof}

\begin{lemma}\label{L4.4}
Suppose the conditions in Lemma \ref{L4.2} are fulfilled, and the ansatz in \eqref{F-1}, \eqref{F-2} and \eqref{F-3} holds for some $k$ with $0 \le k \le d-1$. Then the ansatz  \eqref{F-2} and \eqref{F-3} holds for $k+1$.
\end{lemma}
\begin{proof}
According to Lemma \ref{L4.3}, we know the dynamic of $Q^{k+1}$ is governed by \eqref{F-24}. For the sake of discussion, we rewrite the  differential equation \eqref{F-24} and consider it on $[t_k, + \infty)$,
\begin{equation}\label{F-25}
\dot{Q}^{k+1}(t) \le -\kappa \frac{\sin \gamma}{(\sum_{j=1}^{N-1} (\eta^j A(2N,j)) + 1)\gamma} \left(Q^{k+1}(t) - cD_k(\theta(t)) -\frac{\tilde{c}}{\kappa}\right), \quad t\in [t_k,+\infty),
\end{equation}
where $c$ and $\tilde{c}$ are given in \eqref{F-dd20}. In the following, we will find time $t_{k+1}$ after which the quantity $Q^{k+1}$ in \eqref{F-25} is uniformly bounded. There are two cases we need to consider separately.\newline

\noindent$\bullet$ \textbf{Case 1.} We first consider the case that $Q^{k+1}(t_k) > \frac{\beta^{d-k}D^\infty}{ (4c)^{d-k-1}}$. In this case,  When $Q^{k+1}(t) \in [\frac{\beta^{d-k}D^\infty}{ (4c)^{d-k-1}},Q^{k+1}(t_k)]$, according to \eqref{F-23} and \eqref{F-25}, we have
\begin{equation}\label{F-d25}
\begin{aligned}
\dot{Q}^{k+1}(t) &\le -\kappa \frac{\sin \gamma}{(\sum_{j=1}^{N-1} (\eta^j A(2N,j)) + 1)\gamma} \left(\frac{\beta^{d-k}D^\infty}{ (4c)^{d-k-1}} -\frac{\beta^{d-k}D^\infty}{ 2(4c)^{d-k-1}}\right) \\
&= -\kappa \frac{\sin \gamma}{(\sum_{j=1}^{N-1} (\eta^j A(2N,j)) + 1)\gamma} \frac{\beta^{d-k}D^\infty}{ 2(4c)^{d-k-1}} < 0.
\end{aligned}
\end{equation}
This means that when $Q^{k+1}(t)$ is located in the interval $[\frac{\beta^{d-k}D^\infty}{ (4c)^{d-k-1}},Q^{k+1}(t_k)]$, $Q^{k+1}(t)$ will keep decreasing with a uniform rate.
Therefore, we can define a stopping time $t_{k+1}$ as follows,
\[t_{k+1}=\inf \{t\geq t_k\ |\ Q^{k+1}(t)\le \frac{\beta^{d-k}D^\infty}{ (4c)^{d-k-1}}\}.\]
Then, according to the definition of $t_{k+1}$, we know that $Q^{k+1}$ will decrease before $t_{k+1}$ and has the following property at $t_{k+1}$,
\begin{equation}\label{F-a26}
Q^{k+1}(t_{k+1}) = \frac{\beta^{d-k}D^\infty}{ (4c)^{d-k-1}}.
\end{equation}
Moreover, according to \eqref{F-d25}, it is obvious the stopping time $t_{k+1}$ satisfies the following upper bound estimate, 
\begin{equation}\label{F-26}
t_{k+1} \le \frac{Q^{k+1}(t_k) - \frac{\beta^{d-k}D^\infty}{ (4c)^{d-k-1}}}{\kappa \frac{\sin \gamma}{(\sum_{j=1}^{N-1} (\eta^j A(2N,j)) + 1)\gamma} \frac{\beta^{d-k}D^\infty}{ 2(4c)^{d-k-1}}} + t_k.
\end{equation}
Now we study the upper bound of $Q^{k+1}$ on $[t_{k+1}, +\infty)$.
Coming back to \eqref{F-d25},we can apply \eqref{F-a26} and the same arguments as \eqref{F-a23} to derive  
\begin{equation}\label{F-a31}
Q^{k+1}(t) \le \frac{\beta^{d-k}D^\infty}{ (4c)^{d-k-1}}, \quad t \in [t_{k+1}, +\infty).
\end{equation}

On the other hand, in order to verify \eqref{F-3}, we do further estimates on $t_{k+1}$ in \eqref{F-26}. For the first part on the right-hand side of \eqref{F-26}, according to Lemma \ref{diameter_alpha}, $t_k < \bar{t}$ and the fact that
\begin{equation*}
Q^{k+1}(t_k) \le D_{k+1}(\theta(t_k)) \le D(\theta(t_k)) < \alpha, \quad \frac{\beta^{d-k}D^\infty}{ 2(4c)^{d-k-1}} > \frac{\beta^{d+1}D^\infty}{(4c)^d},
\end{equation*}
we have the following estimates,
\begin{equation}\label{F-27}
\frac{Q^{k+1}(t_k) - \frac{\beta^{d-k}D^\infty}{ (4c)^{d-k-1}}}{\kappa \frac{\sin \gamma}{(\sum_{j=1}^{N-1} (\eta^j A(2N,j)) + 1)\gamma} \frac{\beta^{d-k}D^\infty}{ 2(4c)^{d-k-1}}} < \frac{\alpha}{\kappa \frac{\sin \gamma}{(\sum_{j=1}^{N-1} (\eta^j A(2N,j)) + 1)\gamma} \frac{\beta^{d+1}D^\infty}{(4c)^d} - D(\Omega)}.
\end{equation}
For the term $t_k$ in \eqref{F-26},
based on the assumption \eqref{F-3} for $t_k$, we have
\begin{equation}\label{F-28}
t_k < \frac{(k+1)\alpha}{\kappa \frac{\sin\gamma}{(\sum_{j=1}^{N-1} \eta^j A(2N,j) + 1)\gamma} \frac{\beta^{d+1}D^\infty}{(4c)^d} - D(\Omega)} < \bar{t} = \frac{\alpha - D(\theta(0))}{D(\Omega)}.
\end{equation}
Thus it yields from \eqref{F-26}, \eqref{F-27} and \eqref{F-28}  that the time $t_{k+1}$ can be estimated as below
\begin{equation}\label{F-29}
t_{k+1} < \frac{(k+2)\alpha}{\kappa \frac{\sin\gamma}{(\sum_{j=1}^{N-1} \eta^j A(2N,j) + 1)\gamma} \frac{\beta^{d+1}D^\infty}{(4c)^d} - D(\Omega)}.
\end{equation}
Moreover, according to \eqref{condition_21}, the coupling strength $\kappa$ satisfies the following inequality
\begin{equation}\label{F-30}
\kappa > \left( 1 + \frac{(d+1)\alpha}{\alpha - D(\theta(0))}\right) \frac{(4c)^d \tilde{c}}{\beta^{d+1}D^\infty} > \left( 1 + \frac{(k+2)\alpha}{\alpha - D(\theta(0))}\right) \frac{(4c)^d \tilde{c}}{\beta^{d+1}D^\infty}, \quad 0\le k \le d-1,
\end{equation}
thus we combine \eqref{F-29} and \eqref{F-30} to verify the ansatz \eqref{F-3} for $k+1$ in the first case, i.e.,  the time $t_{k+1}$ has the following estimate,
\begin{equation}\label{F-31}
t_{k+1} < \bar{t} = \frac{\alpha - D(\theta(0))}{D(\Omega)}.
\end{equation}

\noindent $\bullet$ \textbf{Case 2.} For another case that $Q^{k+1}(t_k) \le \frac{\beta^{d-k}D^\infty}{ (4c)^{d-k-1}}$. Similar to the analysis in \eqref{F-a23}, we apply \eqref{F-d25} to conclude that
\begin{equation}\label{F-b31}
Q^{k+1}(t) \le \frac{\beta^{d-k}D^\infty}{ (4c)^{d-k-1}}, \quad t \in [t_{k}, +\infty).
\end{equation}
This allows us to directly set $t_{k+1} = t_k$. Then, according to \eqref{F-28}, we know \eqref{F-29} and \eqref{F-31} hold, which finish the verification of the ansatz \eqref{F-3} in the second case. \newline

Finally, we are ready to verify the ansatz \eqref{F-2} and \eqref{F-3} for $k+1$. Actually, we can apply \eqref{F-a31}, \eqref{F-b31} and Lemma \ref{beta_QD_k} to have the upper bound of $D_{k+1}(\theta)$ on $[t_{k+1}, +\infty)$ as below
\begin{equation}\label{F-32}
\begin{aligned}
D_{k+1}(\theta(t)) \le \frac{Q^{k+1}(t)}{\beta} \le \frac{\beta^{d-k-1}D^\infty}{ (4c)^{d-k-1}}, \quad t \in [t_{k+1}, +\infty),
\end{aligned}
\end{equation}
Then we combine \eqref{F-29}, \eqref{F-31} and \eqref{F-32} in Case 1 and similar analysis in Case 2 to conclude that the Claim \eqref{F-2} and \eqref{F-3} are true for $D^{k+1}(\theta)$.
\end{proof}
\vspace{0.5cm}

\noindent \textbf{Proof of Theorem  \ref{enter_small}:} Now, we are ready to prove the main theorem. Combining Lemma \ref{L4.2}, Lemma \ref{L4.3} and Lemma \ref{L4.4}, we apply inductive criteria to conclude that the ansatz \eqref{F-1} --\eqref{F-3} hold for all $0\leq k\leq d$. Then, we immediately obtain from \eqref{F-2} that there exists time $t_d$ such that
\begin{equation*}
D(\theta(t)) = D_d(\theta(t)) \le D^\infty, \quad \mbox{for} \ t \in [t_d, +\infty),
\end{equation*}
which yields the desired result in Theorem \ref{enter_small}. 

\begin{remark}
In Theorem \ref{enter_small}, we show that the phase diameter will enter into a small region after some finite time, which means $\cos x$ is positive after the finite time. Therefore, we can lift \eqref{KuM} to the second-order formulation, which enjoys the similar form to Cucker-Smale model with the interaction function $\cos x$. 

More precisely, we can introduce phase velocity $\omega_i(t) := \dot{\theta}_i(t)$ for each oscillator, and 
directly differentiate \eqref{KuM} with respect to time $t$ to derive the equivalent Cucker-Smale type second order model as below
\begin{equation}\label{s_KuM}
\begin{cases} 
\displaystyle \dot{\theta}_i(t) = \omega_i(t), \quad t > 0, \quad i =1,2,\ldots,N, \\
\displaystyle \dot{\omega}_i(t) = \kappa \sum_{j \in \mathcal{N}_i} \cos (\theta_j(t) - \theta_i(t)) (\omega_j(t) - \omega_i(t)), \\
\displaystyle \omega_i(0) = \dot{\theta}_i(0).
\end{cases}
\end{equation}
\end{remark}

\begin{corollary}\label{complete_syn}
Let $\theta_i$ be a solution to system \eqref{s_KuM} and suppose the assumptions in Lemma \ref{L4.2} are fulfilled. Moreover, assume that there exists time $t_* > 0$ such that
\begin{equation}\label{small_region}
D(\theta(t)) \le D^\infty, \quad \ t \in [t_*, +\infty),
\end{equation}
where $D^\infty < \frac{\pi}{2}$ is a small positive constant. Then there exist positive constants $C_1$ and $C_2$ such that
\begin{equation*}
D(\omega(t)) \le C_1 e^{-C_2 (t- t_*)}, \quad t > t_*,
\end{equation*}
where $D(\omega(t)) = \max_{1\le i \le N} \{\omega_i(t)\} - \min_{1\le i \le N} \{\omega_i(t)\}$ is the diameter of phase velocity.
\end{corollary}
\begin{proof}
We can apply Theorem \ref{enter_small} and the methods and results in the work of Ha et al. \cite{H-L-Z20} for Cucker-Smale model to yield the emergence of exponentially fast synchronization in \eqref{KuM} and \eqref{s_KuM}. As the proof is almost the same as in \cite{H-L-Z20}, we omit the details, and we refer the readers to \cite{H-L-Z20} for more infomation.\newline 
\end{proof}

\section{Summary}\label{sec:5}
\setcounter{equation}{0}
In this paper, we presented a sufficient framework for the complete synchronization of the Kuramoto model with general network containing a spanning tree. To this end, we followed a node decomposition introduced in \cite{H-L-Z20} to construct new quantities which are equivalent to phase diameters. In a large coupling strength, when the initial data is confined in an open half circle, we proved that the phase diameter of the whole ensemble will concentrate into a small region, thus we can apply the method in \cite{H-L-Z20} or \cite{D-H-K20} which yielded that the complete synchronization occurs exponentially fast. However, our analytical method restricts the initial phase configuration to be confined in a half circle. It would be interesting to see whether the restriction on the initial data can be replaced by a generic one. This interesting issue will be further dealt with in our future work.\newline

\begin{appendix}

\section{proof of Lemma \ref{dynamics_Q0}}\label{appendix1}
We will split the proof into six steps. In the first step, we show that the phase diameter of $\mathcal{G}_0$ is bounded by $\gamma$ in a finite time interval. In the second, third and forth steps, we use induction criteria to construct the differential inequality of $Q^0(t)$ in the finite time interval. In the last two steps, we exploit the derived differential inequality of $Q^0(t)$ to conclude that phase diameter of $\mathcal{G}_0$ is bounded by $\gamma$ on $[0, +\infty)$, and thus the differential inequality of $Q^0(t)$ obtained in second step also holds on $[0, +\infty)$.\newline

\noindent $\bigstar$ \textbf{Step 1.} We first define a set
\begin{equation*}
\mathcal{B}_0 := \{ T >0 : \ D_0(\theta(t)) < \gamma, \ \forall \ t \in [0,T) \}.
\end{equation*}
According to Lemma \ref{t_bar}, the set $\mathcal{B}_0$ is non-empty since 
\[D_0(\theta(t))  < \alpha < \gamma, \quad t \in [0,\bar{t}),\]
which implies that $\bar{t} \in \mathcal{B}_0$. In the following, we set $T^* = \sup \mathcal{B}_0$, and prove $T^* = +\infty$ to finish the proof of the lemma. If not, i.e., suppose $T^* < +\infty$, then we apply the continuity of $D_0(\theta(t))$ to have
\begin{equation}\label{C-0}
D_0(\theta(t)) < \gamma, \quad \forall \ t\in [0,T^*), \quad D_0(\theta(T^*)) = \gamma.
\end{equation}
In particular, we have $\bar{t} \le T^*$. According to the standard theory of ordinary differential equation, the solution to system \eqref{KuM} is analytic. Therefore, in the finite time interval $[0,T^*)$, any two oscillators either collide finite times or always stay together. If there are some $\theta_i$ and $\theta_j$ which always stay together in $[0,T^*]$, we can view them as one oscillator and thus the total number of oscillators that we need to study can be reduced. For this more simpler situation, we can deal with it in a similar method. Therefore, we only consider the case that there is no pair of oscillators staying together in all period $[0,T^*)$. In this situation, only finite many collisions occur through $[0,T^*)$. Thus, we divide the time interval $[0,T^*)$ into a finite union as below
\[[0,T^*) = \bigcup_{l=1}^r J_l, \quad J_l = [t_{l-1},t_l),\]
where the end point $t_l$ denotes the collision time. It is clear that there is no collision in the interior of $J_l$. Then we pick out any time interval $J_l$ and assume that
\begin{equation}\label{C-1}
\theta_1^0(t) \le \theta^0_2(t) \le \ldots \le \theta^0_{N_0}(t), \quad t \in J_l. 
\end{equation}
$\ $

\noindent $\bigstar$ \textbf{Step 2.} According to the notations in \eqref{abbreviation_0}, we follow the process $\mathcal{A}_1$ and $\mathcal{A}_2$ to construct $\bar{\theta}^0_n$ and $\underline{\theta}^0_n, \ 1\le n\le N_0$, respecively.
We first consider the dynamics of $\bar{\theta}_{N_0}^0 = \theta_{N_0}^0$,
\begin{equation}\label{C1-1}
\dot{\theta}_{N_0}^0(t) = \Omega^0_{N_0} + \kappa \sum_{j\in \mathcal{N}_{N_0}^0(0)} \sin (\theta^0_j - \theta^0_{N_0}) \le \Omega_M + \kappa \min_{j \in \mathcal{N}^0_{N_0}(0)} \sin(\theta^0_j - \theta^0_{N_0}).
\end{equation}
The last inequality above holds because of the negative sign of $\sin(\theta^0_j(t) - \theta^0_{N_0}(t))$ due to the well-ordered assumption \eqref{C-1}. For the dynamics of $\bar{\theta}^0_{N_0-1}$, according to the process $\mathcal{A}_1$ and $\bar{a}^0_{N_0-1} = \eta(N_0 + 2)$ in \eqref{coeffi_0}, we have estimates for the derivative of $\bar{\theta}^0_{N_0-1}$ as follows,
\begin{align}
\dot{\bar{\theta}}^0_{N_0-1} &= \frac{d}{dt}\left( \frac{\bar{a}^0_{N_0-1} \theta^0_{N_0} + \theta^0_{N_0-1}}{\bar{a}^0_{N_0-1}+1}\right) = \frac{\bar{a}^0_{N_0-1}}{\bar{a}^0_{N_0-1}+1} \dot{\theta}^0_{N_0} + \frac{1}{\bar{a}^0_{N_0-1}+1} \dot{\theta}^0_{N_0-1}\notag\\
&\le \frac{\bar{a}^0_{N_0-1}}{\bar{a}^0_{N_0-1}+1} \left(\Omega_M + \kappa \min_{j \in \mathcal{N}^0_{N_0}(0)} \sin(\theta^0_j - \theta^0_{N_0})\right) \notag\\
&+ \frac{1}{\bar{a}^0_{N_0-1}+1}\left(\Omega^0_{N_0-1} + \kappa \sum_{j \in \mathcal{N}^0_{N_0-1}(0)} \sin(\theta^0_j - \theta^0_{N_0-1})\right)\notag\\
&\le \Omega_M + \kappa \frac{\eta(N_0 + 2)}{\bar{a}^0_{N_0-1}+1}  \min_{j \in \mathcal{N}^0_{N_0}(0)} \sin(\theta^0_j - \theta^0_{N_0}) \label{C-2}\\
&+ \kappa \frac{1}{\bar{a}^0_{N_0-1}+1} \left( \sin(\theta^0_{N_0} - \theta^0_{N_0-1}) + \underset{j \le N_0-1}{\sum_{j\in \mathcal{N}_{N_0-1}^0(0)}} \sin (\theta^0_j - \theta^0_{N_0})\right) \notag\\
&\le \Omega_M + \kappa \frac{\eta}{\bar{a}^0_{N_0-1}+1} \min_{j \in \mathcal{N}^0_{N_0}(0)} \sin(\theta^0_j - \theta^0_{N_0}) + \kappa \frac{1}{\bar{a}^0_{N_0-1}+1} \underset{j \le N_0-1}{\min_{j\in \mathcal{N}_{N_0-1}^0(0)}} \sin (\theta^0_j - \theta^0_{N_0-1}) \notag\\
&+ \kappa \frac{1}{\bar{a}^0_{N_0-1}+1} \underbrace{\left(\eta \min_{j \in \mathcal{N}^0_{N_0}(0)} \sin(\theta^0_j - \theta^0_{N_0}) + \sin(\theta^0_{N_0} - \theta^0_{N_0-1})\right)}_{\mathcal{I}_2}.\notag
\end{align}
We now show the term $\mathcal{I}_2$ is non-positive. We will only consider the situation $\gamma > \frac{\pi}{2}$, and the situation $\gamma \le \frac{\pi}{2}$ can be similarly dealt with. It is obvious that
\[\min_{j \in \mathcal{N}^0_{N_0}(0)} \sin(\theta^0_j - \theta^0_{N_0})  \le \sin(\theta^0_{\bar{k}_{N_0}} - \theta^0_{N_0}), \quad \bar{k}_{N_0} = \min_{j \in  \mathcal{N}^0_{N_0}(0)} j. \]
Note that $\bar{k}_{N_0} \le N_0-1$ since $\bar{\mathcal{L}}^{N_0}_{N_0}(\bar{C}_{N_0,N_0})$ is not a general root. Therefore, if $ 0 \le \theta^0_{N_0}(t) - \theta^0_{\bar{k}_{N_0}}(t) \le \frac{\pi}{2}$, we immediately obtain that
\begin{equation*}
0 \le \theta^0_{N_0}(t) - \theta^0_{N_0-1}(t) \le \theta^0_{N_0}(t) - \theta^0_{\bar{k}_{N_0}}(t) \le \frac{\pi}{2}, 
\end{equation*}
which implies that
\begin{equation*}
\mathcal{I}_2 \le \eta \sin(\theta^0_{\bar{k}_{N_0}} - \theta^0_{N_0}) + \sin(\theta^0_{N_0} - \theta^0_{N_0-1}) \le \sin(\theta^0_{\bar{k}_{N_0}} - \theta^0_{N_0}) + \sin(\theta^0_{N_0} - \theta^0_{N_0-1}) \le 0.
\end{equation*}
On the other hand, if $\frac{\pi}{2} <  \theta^0_{N_0}(t) - \theta^0_{\bar{k}_{N_0}}(t) < \gamma$, we use the fact
\[\eta > \frac{1}{\sin \gamma} \quad \mbox{and} \quad \sin(\theta^0_{N_0}(t) - \theta^0_{\bar{k}_{N_0}}(t)) > \sin \gamma,\] 
to conclude that $\eta \sin(\theta^0_{\bar{k}_{N_0}} - \theta^0_{N_0}) \le -1$. Hence, in this case, we still obtain that
\begin{equation*}
\mathcal{I}_2 \le \eta \sin(\theta^0_{\bar{k}_{N_0}} - \theta^0_{N_0}) + \sin(\theta^0_{N_0} - \theta^0_{N_0-1}) \le -1 + 1 \le 0.
\end{equation*}
Thus, for $t\in J_l$, we combine above analysis to conclude that 
\begin{equation}\label{C-3}
\mathcal{I}_2 = \eta \min_{j \in \mathcal{N}^0_{N_0}(0)} \sin(\theta^0_j - \theta^0_{N_0}) + \sin(\theta^0_{N_0} - \theta^0_{N_0-1}) \le 0.
\end{equation} 
 Then combining \eqref{C-2} and \eqref{C-3}, we derive that
\begin{equation}\label{C-4}
\dot{\bar{\theta}}^0_{N_0-1} \le \Omega_M + \kappa \frac{1}{\bar{a}^0_{N_0-1}+1} \left(\eta \min_{j \in \mathcal{N}^0_{N_0}(0)} \sin(\theta^0_j - \theta^0_{N_0}) + \underset{j \le N_0-1}{\min_{j\in \mathcal{N}_{N_0-1}^0(0)}} \sin (\theta^0_j - \theta^0_{N_0-1}) \right).
\end{equation}
$\ $

\noindent $\bigstar$ \textbf{Step 3.} Now we apply the induction principle to cope with $\bar{\theta}^0_n$ in \eqref{abbreviation_0},  which are construced in the iteration process $\mathcal{A}_1$. We will prove for $1 \le n \le N_0$ that,
\begin{equation}\label{C-5}
\dot{\bar{\theta}}^0_{n}(t) \le \Omega_M + \kappa \frac{1}{\bar{a}^0_{n}+1} \sum_{i=n}^{N_0}\left( \eta^{i-n} \underset{j \le i}{\min_{j\in \mathcal{N}_{i}^0(0)}} \sin (\theta^0_j(t) - \theta^0_{i}(t))\right).
\end{equation}
In fact, \eqref{C-5} already holds for $n = N_0, N_0-1$ from \eqref{C1-1} and \eqref{C-4}. Then, suppose that for $n\le l \le N_0$ where $2 \le n \le N_0$, we have
\begin{equation}\label{C-6}
\dot{\bar{\theta}}^0_{l}(t) \le \Omega_M + \kappa \frac{1}{\bar{a}^0_{l}+1} \sum_{i=l}^{N_0}\left( \eta^{i-l} \underset{j \le i}{\min_{j\in \mathcal{N}_{i}^0(0)}} \sin (\theta^0_j(t) - \theta^0_{i}(t))\right),
\end{equation}
we next verify that \eqref{C-5} still holds for $l = n-1$. 
According to the Algorithm $\mathcal{A}_1$ and \eqref{C-6}, the dynamics of the quantity $\bar{\theta}^0_{n-1}(t)$ has following estimates,

\begin{align}
\dot{\bar{\theta}}^0_{n-1} &= \frac{d}{dt}\left( \frac{\bar{a}^0_{n-1} \bar{\theta}^0_{n} + \theta^0_{n-1}}{\bar{a}^0_{n-1}+1}\right) = \frac{\bar{a}^0_{n-1}}{\bar{a}^0_{n-1}+1} \dot{\bar{\theta}}^0_{n} + \frac{1}{\bar{a}^0_{n-1}+1} \dot{\theta}^0_{n-1} \notag\\
&\le \frac{\bar{a}^0_{n-1}}{\bar{a}^0_{n-1}+1} \left[\Omega_M + \kappa \frac{1}{\bar{a}^0_{n}+1} \sum_{i=n}^{N_0}\left( \eta^{i-n} \underset{j \le i}{\min_{j\in \mathcal{N}_{i}^0(0)}} \sin (\theta^0_j - \theta^0_{i})\right)\right]\notag\\
&+ \frac{1}{\bar{a}^0_{n-1}+1} \left( \Omega^0_{n-1} + \kappa \sum_{j \in \mathcal{N}^0_{n-1}(0)} \sin (\theta^0_j - \theta^0_{n-1})\right)\notag\\
&\le \Omega_M + \kappa \frac{\eta(2N_0-n+2)}{\bar{a}^0_{n-1}+1} \sum_{i=n}^{N_0}\left( \eta^{i-n} \underset{j \le i}{\min_{j\in \mathcal{N}_{i}^0(0)}} \sin (\theta^0_j - \theta^0_{i})\right)\notag \\
&+ \kappa \frac{1}{\bar{a}^0_{n-1}+1} \left( \underset{j \le n-1}{\sum_{j \in \mathcal{N}^0_{n-1}(0)}} \sin (\theta^0_j - \theta^0_{n-1}) + \underset{j > n-1}{\sum_{j \in \mathcal{N}^0_{n-1}(0)}} \sin (\theta^0_j - \theta^0_{n-1})\right) \notag\\
&\le \Omega_M + \underbrace{\kappa \frac{\eta N_0}{\bar{a}^0_{n-1}+1} \sum_{i=n}^{N_0}\left( \eta^{i-n} \underset{j \le i}{\min_{j\in \mathcal{N}_{i}^0(0)}} \sin (\theta^0_j - \theta^0_{i})\right) }_{\mathcal{I}_3} \label{C-7}\\
&+  \kappa \frac{\eta}{\bar{a}^0_{n-1}+1} \sum_{i=n}^{N_0}\left( \eta^{i-n} \underset{j \le i}{\min_{j\in \mathcal{N}_{i}^0(0)}} \sin (\theta^0_j - \theta^0_{i})\right) +  \kappa \frac{1}{\bar{a}^0_{n-1}+1} \underset{j \le n-1}{\min_{j\in \mathcal{N}_{n-1}^0(0)}} \sin (\theta^0_j - \theta^0_{n-1}) \notag\\
&+  \frac{\kappa}{\bar{a}^0_{n-1}+1} \left(\underbrace{\eta(N_0-n+1) \sum_{i=n}^{N_0}\left( \eta^{i-n} \underset{j \le i}{\min_{j\in \mathcal{N}_{i}^0(0)}} \sin (\theta^0_j - \theta^0_{i})\right) + \underset{j > n-1}{\sum_{j \in \mathcal{N}^0_{n-1}(0)}} \sin (\theta^0_j - \theta^0_{n-1})}_{\mathcal{I}_4}\right).\notag
\end{align}
\vspace{1cm}

\noindent In above estimates, we used the fact that
\[\bar{a}^0_{n-1} = \eta(2N_0-n + 2)(\bar{a}^0_{n} +1), \quad \underset{j \le n-1}{\sum_{j \in \mathcal{N}^0_{n-1}(0)}} \sin (\theta^0_j - \theta^0_{n-1}) \le \underset{j \le n-1}{\min_{j\in \mathcal{N}_{n-1}^0(0)}} \sin (\theta^0_j - \theta^0_{n-1}).\]
It is obvious that $\mathcal{I}_3 \le 0$, and thus we can neglect it. In the subsequence, we will deal with $\mathcal{I}_4$ and prove that
\begin{equation}\label{C-8}
\mathcal{I}_4 \le 0.
\end{equation}
In fact, according to Lemma \ref{eta_sin_inequality}, we directly have
\begin{equation}\label{C-9}
\sum_{i=n}^{N_0}\left( \eta^{i-n} \underset{j \le i}{\min_{j\in \mathcal{N}_{i}^0(0)}} \sin (\theta^0_j - \theta^0_{i})\right) \le \sin(\theta^0_{\bar{k}_{n}} - \theta^0_{N_0}), \quad \bar{k}_{n} = \min_{j \in \bigcup_{i=n}^{N_0} \mathcal{N}^0_{i}(0)} j.
\end{equation}
Similar to the analysis in \eqref{C-3}, we only deal with $\mathcal{I}_4$ under the situation $\gamma > \frac{\pi}{2}$. Now we consider two cases    according to the relation of size between $\theta^0_{N_0} - \theta^0_{\bar{k}_{n}}$ and $\frac{\pi}{2}$.\newline
\noindent$\diamond$ For the case that $ 0 \le \theta^0_{N_0}(t) - \theta^0_{\bar{k}_{n}}(t) \le \frac{\pi}{2}$, we immediately obtain that for $j \in \mathcal{N}^0_{n-1}(0), \ j>n-1$,
\begin{equation*}
0 \le \theta^0_{j}(t) - \theta^0_{n-1}(t) \le \theta^0_{N_0}(t) - \theta^0_{n-1}(t) \le \theta^0_{N_0}(t) - \theta^0_{\bar{k}_{n}}(t) \le \frac{\pi}{2},
\end{equation*}
where we use the fact that $\bar{k}_{n} \le n-1$ as $\bar{\mathcal{L}}^{N_0}_{n}(\bar{C}_{n,N_0})$ is not a general root.
Then, we combine \eqref{C-9} to have
\begin{equation*}
\begin{aligned}
\mathcal{I}_4 &\le \eta(N_0-n+1) \sin(\theta^0_{\bar{k}_{n}} - \theta^0_{N_0}) + \underset{j > n-1}{\sum_{j \in \mathcal{N}^0_{n-1}(0)}} \sin (\theta^0_j - \theta^0_{n-1}) \\
&\le (N_0-n+1) \sin(\theta^0_{\bar{k}_{n}} - \theta^0_{N_0})  + (N_0-n+1) \sin(\theta^0_{N_0}(t) - \theta^0_{n-1}(t)) \le 0,
\end{aligned}
\end{equation*}
where we apply the fact $\eta > 1$ and the monotone property of sine function in $[0,\frac{\pi}{2}]$.\newline
\noindent $\diamond$For another case that $\frac{\pi}{2} <  \theta^0_{N_0}(t) - \theta^0_{\bar{k}_{n}}(t) < \gamma$, it is known that
\[\eta > \frac{1}{\sin \gamma} \quad \mbox{and} \quad \sin(\theta^0_{N_0}(t) - \theta^0_{\bar{k}_{n}}(t)) > \sin \gamma,\] 
which means  $\eta \sin(\theta^0_{\bar{k}_{n}} - \theta^0_{N_0}) \le -1$. Thus we obtain that
\begin{equation*}
\begin{aligned}
\mathcal{I}_4 &\le \eta(N_0-n+1) \sin(\theta^0_{\bar{k}_{n}} - \theta^0_{N_0}) + \underset{j > n-1}{\sum_{j \in \mathcal{N}^0_{n-1}(0)}} \sin (\theta^0_j - \theta^0_{n-1}) \\
&\le -(N_0-n+1) + (N_0 -n + 1) = 0.
\end{aligned}
\end{equation*}
Therefore, \eqref{C-8} holds at time $t \in J_l$. Now we combine \eqref{C-7} and \eqref{C-8} to get 
\begin{equation*}
\begin{aligned}
\dot{\bar{\theta}}^0_{n-1} &\le \Omega_M+ \kappa \frac{1}{\bar{a}^0_{n-1}+1}  \left[\sum_{i=n}^{N_0}\left( \eta^{i-(n-1)} \underset{j \le i}{\min_{j\in \mathcal{N}_{i}^0(0)}} \sin (\theta^0_j - \theta^0_{i})\right) +  \underset{j \le n-1}{\min_{j\in \mathcal{N}_{n-1}^0(0)}} \sin (\theta^0_j - \theta^0_{n-1})\right] \\
&= \Omega_M + \kappa \frac{1}{\bar{a}^0_{n-1}+1} \sum_{i=n-1}^{N_0}\left( \eta^{i-(n-1)} \underset{j \le i}{\min_{j\in \mathcal{N}_{i}^0(0)}} \sin (\theta^0_j - \theta^0_{i})\right).
\end{aligned}
\end{equation*}
So far, we complete the proof of the claim \eqref{C-5}.\newline 

\noindent $\bigstar$ \textbf{Step 4.}
Now, we set $n=1$ in \eqref{C-5} and apply Lemma \ref{eta_sin_inequality} to have
\begin{equation}\label{C-10}
\begin{aligned}
\dot{\bar{\theta}}^0_{1}(t) &\le \Omega_M + \kappa \frac{1}{\bar{a}^0_{1}+1} \sum_{i=1}^{N_0}\left( \eta^{i-1} \underset{j \le i}{\min_{j\in \mathcal{N}_{i}^0(0)}} \sin (\theta^0_j(t) - \theta^0_{i}(t))\right)\\
&\le \Omega_M + \kappa \frac{1}{\bar{a}^0_{1}+1} \sin(\theta^0_{\bar{k}_{1}} - \theta^0_{N_0}) = \Omega_M + \kappa \frac{1}{\bar{a}^0_{1}+1} \sin(\theta^0_1 - \theta^0_{N_0}),
\end{aligned}
\end{equation}
where $\bar{k}_{1} = \min_{j \in \bigcup_{i=1}^{N_0} \mathcal{N}^0_{i}(0)} j =1$ due to the strong connectivity of $\mathcal{G}_0$. Similarly, we can follow the process $\mathcal{A}_2$ to construct $\underline{\theta}^0_k$ in \eqref{abbreviation_0} until $k = N_0$. Then, we can apply the similar argument as before to obtain that,
\begin{equation}\label{C-11}
\begin{aligned}
\frac{d}{dt} \underline{\theta}^0_{N_0}(t) &\ge \Omega_m + \kappa \frac{1}{\underline{a}^0_{N_0} + 1} \sum_{i=1}^{N_0}\left( \eta^{N_0-i} \underset{j \ge i}{\max_{j\in \mathcal{N}_{i}^0(0)}} \sin (\theta^0_j(t) - \theta^0_{i}(t))\right) \\
&\ge \Omega_m + \kappa \frac{1}{\underline{a}^0_{N_0} + 1} \sin (\theta^0_{\underline{k}_{N_0}} - \theta^0_1) = \Omega_m + \kappa \frac{1}{\bar{a}^0_{1} + 1} \sin (\theta^0_{N_0} - \theta^0_1),
\end{aligned}
\end{equation}
where we use the strong connectivity and the fact that $\underline{k}_{N_0} = \max_{j \in \bigcup_{i=1}^{N_0} \mathcal{N}^0_{i}(0)} j =N_0$ and $\underline{a}^0_{N_0}=\bar{a}^0_{1}$. Then we recall the notations $\bar{\theta}_0 = \bar{\theta}^0_1$ and $\underline{\theta}_0 = \underline{\theta}^0_{N_0}$, and combine \eqref{C-10} and \eqref{C-11} to obtain that
\begin{equation*}\label{C-12}
\begin{aligned}
\dot{Q}^0(t) &= \frac{d}{dt} (\bar{\theta}_0 - \underline{\theta}_0) \le D(\Omega) - \kappa \frac{2}{\bar{a}^0_{1} + 1} \sin (\theta^0_{N_0} - \theta^0_1) \\
&\le D(\Omega) - \kappa \frac{1}{ \sum_{j=1}^{N_0-1} (\eta^j A(2N_0,j)) + 1} \sin (\theta^0_{N_0} - \theta^0_1),
\end{aligned}
\end{equation*}
In the above estimates, we use the property
\[\bar{a}^0_{1} = \sum_{j=1}^{N_0-1} (\eta^j A(2N_0,j)) .\]
Since the function  $\frac{\sin x}{x}$ is monotonically decreasing in $(0, \pi]$, we apply \eqref{C-0} to obtain  that
\[\sin (\theta^0_{N_0} - \theta^0_1) \ge \frac{\sin \gamma}{\gamma}(\theta^0_{N_0} - \theta^0_1).\]
Moreover, due to the formula $Q^0(t) \le \theta^0_{N_0}(t) - \theta^0_1(t)$, we have
\begin{equation}\label{C-12}
\begin{aligned}
\dot{Q}^0(t) &\le D(\Omega) - \kappa \frac{1}{ \sum_{j=1}^{N_0-1} (\eta^j A(2N_0,j)) + 1} \frac{\sin \gamma}{\gamma} (\theta^0_{N_0} - \theta^0_1) \\
&\le D(\Omega) - \kappa \frac{1}{ \sum_{j=1}^{N_0-1} (\eta^j A(2N_0,j)) + 1} \frac{\sin \gamma}{\gamma} Q^0(t), \quad t \in J_l.
\end{aligned}
\end{equation}
Note that the constructed quantity $Q^0(t) = \bar{\theta}_0(t) - \underline{\theta}_0(t)$ is Lipschitz continuous on $[0,T^*)$. 
Moreover, the above analysis does not depend on the time interval $J_l, \ l=1,2,\ldots,r$, thus the differential inequality \eqref{C-12} holds almost everywhere on $[0,T^*)$. \newline

\noindent $\bigstar$ \textbf{Step 5.} For a given sufficiently small $D^\infty < \min\{\frac{\pi}{2},\alpha\}$, based on the assumption of the coupling strength in \eqref{condition_3}, we have
\begin{equation}\label{C-13}
\kappa > \left( 1 + \frac{\alpha}{\alpha - D(\theta(0))}\right) \frac{\tilde{c}}{\beta D^\infty} > \frac{1}{\beta D^\infty} \frac{D(\Omega) (\sum_{j=1}^{N_0-1}\eta^j A(2N_0,j) + 1)\gamma}{\sin \gamma}
\end{equation}
where
\[\tilde{c} = \frac{D(\Omega) (\sum_{j=1}^{N_0-1}\eta^j A(2N_0,j) + 1)\gamma}{\sin \gamma}.\]
Next we study the upper bound of $Q^0(t)$ in the period $[0,T^*)$. Define
\begin{equation*}
M_0 = \max\left\{Q^0(0), \beta D^\infty\right\}.
\end{equation*}
We claim that
\begin{equation}\label{C-14}
Q^0(t) \le M_0 \quad \mbox{for all} \ t\in [0,T^*).
\end{equation}
Suppose not, then there exists some $\tilde{t} \in [0,T^*)$ such that $Q^0(\tilde{t}) > M_0$. We construct a set
\[\mathcal{C}_0  := \{t < \tilde{t} \ | \ Q^0(t) \le M_0\} .\]
Since $0 \in \mathcal{C}_0$, the set $\mathcal{C}_0$ is not empty. Then we denote $t^* = \sup \mathcal{C}_0$, and immediately obtain that 
\begin{equation}\label{C-15}
t^* < \tilde{t}, \quad Q^0(t^*) = M_0, \quad Q^0(t) > M_0 \quad \mbox{for} \ t \in (t^*, \tilde{t}].
\end{equation}
According to the construction of $M_0$, \eqref{C-13} and \eqref{C-15} , it is known the following estimates hold  for $t \in (t^*, \tilde{t}]$,
\begin{equation*}
\begin{aligned}
&D(\Omega) - \kappa \frac{1}{ \sum_{j=1}^{N_0-1} (\eta^j A(2N_0,j)) + 1} \frac{\sin \gamma}{\gamma} Q^0(t) \\
&< D(\Omega) - \kappa \frac{1}{ \sum_{j=1}^{N_0-1} (\eta^j A(2N_0,j)) + 1} \frac{\sin \gamma}{\gamma} \beta D^\infty < 0.
\end{aligned}
\end{equation*}
Then, we apply the above inequality and integrate on the both sides of \eqref{C-12} from $t^*$ to $\tilde{t}$ to get
\[0 < Q^0(\tilde{t}) - M_0 = Q^0(\tilde{t}) - Q^0(t^*) \le  \int_{t^*}^{\tilde{t}}( D(\Omega) - \kappa \frac{1}{ \sum_{j=1}^{N_0-1} (\eta^j A(2N_0,j)) + 1} \frac{\sin \gamma}{\gamma} Q^0(t) )dt< 0,\]
which is an obvious contradiction. Thus we complete the proof of \eqref{C-14}.\newline

\noindent $\bigstar$ \textbf{Step 6.}
Now, we are ready to show the contradiction to \eqref{C-0}, and thus it implies that $T^*=+\infty$. In fact, due to the fact that $\beta < 1, D^\infty < \alpha$ and $Q^0(0) \le D_0(\theta(0)) < \alpha$, we have
\[Q^0(t) \le M_0 = \max\left\{Q^0(0), \beta D^\infty \right\} < \alpha, \quad t \in [0,T^*). \]
Then we apply the relation $\beta D_0(\theta(t)) \le Q^0(t)$ given in Lemma \ref{beta_QD} and the assumption $\eta > \frac{2}{1 - \frac{\alpha}{\gamma}}$ in \eqref{condition_3} to obtain that
\[D_0(\theta(t)) \le \frac{Q^0(t)}{\beta} < \frac{\alpha}{\beta} < \gamma, \quad t \in [0,T^*) \quad \mbox{where} \ \beta = 1 - \frac{2}{\eta}.\]
As $D_0(\theta(t))$ is continuous, we have
\begin{equation*}
D_0(\theta(T^*)) = \lim_{t \to (T^*)^-}D_0(\theta(t)) \le \frac{\alpha}{\beta} < \gamma,
\end{equation*}
which contradicts to the situation that $D_0(\theta(T^*)) = \gamma$ in \eqref{C-0}.
Therefore, we derive that $T^* = +\infty$, which yields that
\begin{equation}\label{C-a15}
D_0(\theta(t)) < \gamma, \quad \mbox{for all} \ t \in [0, +\infty).
\end{equation}
Then for any finite time $T>0$, we apply \eqref{C-a15} and repeat the same argument in the second, third, forth steps to obtain the dynamics of $Q^0(t)$ in \eqref{C-12} holds on $[0,T)$. This yields the following differential inequality of $Q^0$ on the whole time interval:
\[\dot{Q}^0(t) \le D(\Omega) - \kappa \frac{1}{ \sum_{j=1}^{N_0-1} (\eta^j A(2N_0,j)) + 1} \frac{\sin \gamma}{\gamma} Q^0(t), \quad t \in [0,+\infty).\]
Thus, we complete the proof of this Lemma. 
\qed
\newline

\section{proof of step 1 in lemma \ref{L4.3}}\label{appendix 2}
\setcounter{equation}{0}
We will show the detailed proof of Step 1 in Lemma \ref{L4.3}. Now we pick out any interval $J_l$ with $1\le l \le r$, where the orders of both $\{\bar{\theta}_i\}_{i=0}^{k+1}$ and $\{\underline{\theta}_i\}_{i=0}^{k+1}$ are preseved and the order of oscillators in each subdigraph $\mathcal{G}_i$ with $0\le i \le k+1$ will not change in each time interval. Then, we consider four cases according to the possibility of relative position between $\bigcup_{i=0}^k \mathcal{G}$ and $\mathcal{G}_{k+1}$. 

\subsection{Case 1} Consider the case that
\begin{equation*}
\max_{0 \le i\le k+1}\{\bar{\theta}_i\}  = \max_{0 \le i\le k}\{\bar{\theta}_i\}, \quad \min_{0\le i\le k+1}\{\underline{\theta}_i\} = \min_{0\le i\le k}\{\underline{\theta}_i\} \quad \mbox{in} \ J_l. 
\end{equation*}
\begin{figure}[h]
\centering
\includegraphics[width=0.5\textwidth]{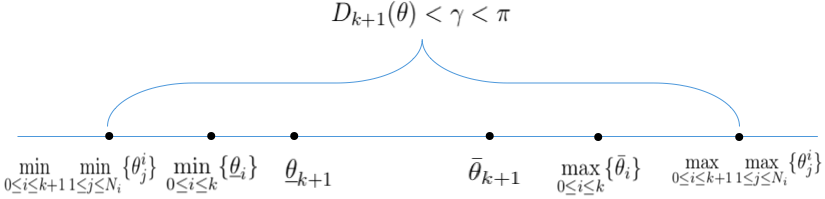}
\caption{The comparison relation in Case 1}
\label{Fig1}
\end{figure}
The comparison relation in this case is showed in Figure \ref{Fig1}.
In this case, $Q^{k+1}(t) = Q^k(t)$, by the assumption of induction principle and from \eqref{F-1}, we obviously have
\begin{equation*}
\begin{aligned}
\frac{d}{dt} Q^{k+1}(t) &= \frac{d}{dt} Q^k(t), \qquad \qquad t \in J_l, \\
&\le D(\Omega) - \frac{\kappa}{\sum_{j=1}^{N-1} (\eta^jA(2N,j)) +1}\frac{\sin \gamma}{\gamma} Q^{k}(t) + \kappa(2N+1)D_{k-1}(\theta(t))\\
&\le D(\Omega) - \frac{\kappa}{\sum_{j=1}^{N-1} (\eta^jA(2N,j)) +1}\frac{\sin \gamma}{\gamma} Q^{k+1}(t) + \kappa(2N+1)D_{k}(\theta(t)),
\end{aligned}
\end{equation*}
where we use $D_{k-1}(\theta(t)) \le D_{k}(\theta(t))$. Thus we obtain the dynamics for $Q^{k+1}(t)$ in \eqref{F-1} on $J_l$.

\subsection{Case 2} Consider the case that
\begin{equation*}
\max_{0 \le i\le k+1}\{\bar{\theta}_i\}  = \bar{\theta}_{k+1} , \quad \min_{0\le i\le k+1}\{\underline{\theta}_i\} =  \underline{\theta}_{k+1} \quad \mbox{in} \ J_l.
\end{equation*}
\begin{figure}[h]
\centering
\includegraphics[width=0.5\textwidth]{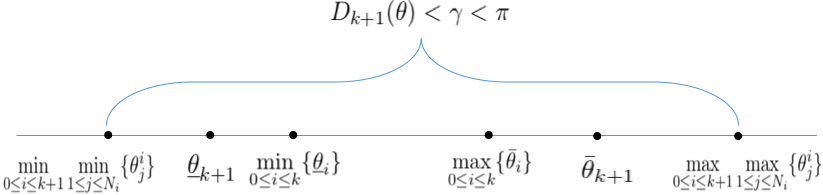}
\caption{The comparison relation in Case 2}
\label{Fig2}
\end{figure}
The comparison relation in this case is presented in Figure \ref{Fig2}.
For this case, we assume that
\begin{equation*}
\theta^{k+1}_1 \le \theta^{k+1}_2 \le \dots \le \theta^{k+1}_{N_{k+1}}, \quad \mbox{on} \ J_l.
\end{equation*}

\noindent $\bigstar$ \textbf{Step 1.} Similar to formula \eqref{C-5}, we claim that for $1 \le n \le N_{k+1}$, the following inequalities hold
\begin{equation}\label{F-5}
\begin{aligned}
\frac{d}{dt} \bar{\theta}^{k+1}_n(t) \le & \ \Omega_M + \kappa S_k D_k(\theta(t))\\
&+  \kappa \frac{1}{\bar{a}^{k+1}_{n}+1} \sum_{i=n}^{N_{k+1}}\left( \eta^{i-n} \underset{j \le i}{\min_{j\in \mathcal{N}_{i}^{k+1}(k+1)}} \sin (\theta^{k+1}_j(t) - \theta^{k+1}_{i}(t))\right) ,
\end{aligned}
\end{equation}
where $S_k = \sum_{i=0}^k N_i$. In the subsequence, we will prove the claim \eqref{F-5} by induction.\newline

\noindent $\bigstar$ \textbf{Step 1.1.} As an initial step, we first verify that \eqref{F-5} holds for $n = N_{k+1}$. In fact, the dynamics of $\bar{\theta}^{k+1}_{N_{k+1}}$ is given by
\begin{equation}\label{F-d6}
\begin{aligned}
\frac{d}{dt} \bar{\theta}^{k+1}_{N_{k+1}} = \frac{d}{dt} \theta^{k+1}  _{N_{k+1}} =& \ \Omega^{k+1}_{N_{k+1}} + \underbrace{\kappa \sum_{j\in \mathcal{N}_{N_{k+1}}^{k+1}(k+1)} \sin (\theta^{k+1}_j - \theta^{k+1}_{N_{k+1}})}_{\mathcal{I}_1} \\
& + \underbrace{\kappa \sum_{l=0}^k \sum_{j\in \mathcal{N}_{N_{k+1}}^{k+1}(l)} \sin (\theta^{l}_j - \theta^{k+1}_{N_{k+1}})}_{\mathcal{I}_2}.
\end{aligned}
\end{equation}
\noindent $\diamond$ \textbf{Estimates on $\mathbf{\mathcal{I}_1}$ in \eqref{F-d6}.} We know that $\theta^{k+1}  _{N_{k+1}}$ is the largest phase among $\mathcal{G}_{k+1}$, and all the oscillators in $\bigcup_{i=0}^{k+1} \mathcal{G}_{i}$ stay in half circle before $T^*$. Therefore, it is clear that
\begin{equation*}
\sin (\theta^{k+1}_j - \theta^{k+1}_{N_{k+1}}) \le 0, \quad \mbox{for} \ j \in \mathcal{N}_{N_{k+1}}^{k+1}(k+1).
\end{equation*}
Then we have 
\begin{equation}\label{F-6}
\sum_{j\in \mathcal{N}_{N_{k+1}}^{k+1}(k+1)} \sin (\theta^{k+1}_j - \theta^{k+1}_{N_{k+1}}) \le \min_{j \in \mathcal{N}^{k+1}_{N_{k+1}}(k+1)} \sin (\theta^{k+1}_j - \theta^{k+1}_{N_{k+1}}).
\end{equation}

\noindent $\diamond$ \textbf{Estimates on $\mathbf{\mathcal{I}_2}$ in \eqref{F-d6}.} For $\theta^l_j$ which is the neighbor of $\theta^{k+1}_{N_{k+1}}$ in $\mathcal{G}_l$ with $0 \le l \le k$, i.e., $j\in \mathcal{N}_{N_{k+1}}^{k+1}(l)$, there are two possible orderings between  $\theta^l_j$ and $\theta^{k+1}_{N_{k+1}}$:

\noindent If $\theta^l_j \le \theta^{k+1}_{N_{k+1}}$, we immediately have
\begin{equation*}
\sin (\theta^{l}_j - \theta^{k+1}_{N_{k+1}}) \le 0.
\end{equation*}
\noindent If $\theta^l_j > \theta^{k+1}_{N_{k+1}}$, according to the fact that
\begin{equation}\label{F-d7}
 \theta^{i}_{N_i} \ge \bar{\theta}_i \ge \underline{\theta}_i \ge \theta^i_1, \quad 0 \le i \le d,
 \end{equation}
we immediately obtain 
\begin{equation}\label{F-7}
\theta^{k+1}_{N_{k+1}} \ge \bar{\theta}_{k+1}  = \max_{0 \le i\le k+1}\{\bar{\theta}_i\}  \ge \max_{0 \le i\le k}\{\bar{\theta}_i\} \ge \min_{0 \le i\le k}\{\underline{\theta}_i\} \ge \min_{0 \le i \le k}\min_{1 \le j \le N_i} \{\theta^i_j\}.
\end{equation}
Thus we use the property of $\sin x  \le x, \ x \ge 0$ and \eqref{F-7} to get
\begin{equation}\label{F-d8}
\sin (\theta^{l}_j - \theta^{k+1}_{N_{k+1}}) \le \theta^{l}_j - \theta^{k+1}_{N_{k+1}} \le \theta^{l}_j -  \min_{0 \le i \le k}\min_{1 \le j \le N_i} \{\theta^i_j\} \le D_k(\theta(t)).
\end{equation}
Therefore, combining the above discussion, \eqref{F-d6} and \eqref{F-6}, we obtain that
\begin{equation*}
\frac{d}{dt} \bar{\theta}^{k+1}_{N_{k+1}} \le \Omega_M + \kappa \min_{j \in \mathcal{N}^{k+1}_{N_{k+1}}(k+1)} \sin (\theta^{k+1}_j - \theta^{k+1}_{N_{k+1}}) + \kappa S_k D_k(\theta(t)).
\end{equation*}
Thus we have that \eqref{F-5} holds for $n = N_{k+1}$.\newline

\noindent $\bigstar$ \textbf{Step 1.2.} Next, we will apply inductive criteria. We assume that \eqref{F-5} holds for $n$ with $2 \le n \le N_{k+1}$, and we  will show that \eqref{F-5} holds for $n-1$. According to the process $\mathcal{A}_1$, we have

\begin{equation}\label{F-8}
\begin{aligned}
\dot{\bar{\theta}}^{k+1}_{n-1} &= \frac{d}{dt}\left( \frac{\bar{a}^{k+1}_{n-1} \bar{\theta}^{k+1}_{n} + \theta^{k+1}_{n-1}}{\bar{a}^{k+1}_{n-1}+1}\right) = \frac{\bar{a}^{k+1}_{n-1}}{\bar{a}^{k+1}_{n-1}+1} \dot{\bar{\theta}}^{k+1}_{n} + \frac{1}{\bar{a}^{k+1}_{n-1}+1} \dot{\theta}^{k+1}_{n-1} \\
&\le  \frac{\bar{a}^{k+1}_{n-1}}{\bar{a}^{k+1}_{n-1}+1} \left[\Omega_M +  \kappa \frac{1}{\bar{a}^{k+1}_{n}+1} \sum_{i=n}^{N_{k+1}}\left( \eta^{i-n} \underset{j \le i}{\min_{j\in \mathcal{N}_{i}^{k+1}(k+1)}} \sin (\theta^{k+1}_j(t) - \theta^{k+1}_{i}(t))\right) \right] \\
&+ \frac{\bar{a}^{k+1}_{n-1}}{\bar{a}^{k+1}_{n-1}+1}  \kappa S_k D_k(\theta(t)) \\
& + \frac{1}{\bar{a}^{k+1}_{n-1}+1} \left( \Omega^{k+1}_{n-1} + \kappa \sum_{j\in \mathcal{N}_{n-1}^{k+1}(k+1)}\sin (\theta^{k+1}_j - \theta^{k+1}_{n-1})  + \kappa \sum_{l=0}^k \sum_{j\in \mathcal{N}_{n-1}^{k+1}(l)} \sin (\theta^{l}_j - \theta^{k+1}_{n-1})
\right) \\
&\le \Omega_M + \underbrace{\kappa \frac{\eta(2N - N_{k+1} - S_k)}{\bar{a}^{k+1}_{n-1}+1} \sum_{i=n}^{N_{k+1}}\left( \eta^{i-n} \underset{j \le i}{\min_{j\in \mathcal{N}_{i}^{k+1}(k+1)}} \sin (\theta^{k+1}_j(t) - \theta^{k+1}_{i}(t))\right)}_{\mathcal{I}_{11}}\\
&+ \kappa \frac{\eta(N_{k+1} -n+2+S_k)}{\bar{a}^{k+1}_{n-1}+1} \sum_{i=n}^{N_{k+1}}\left( \eta^{i-n} \underset{j \le i}{\min_{j\in \mathcal{N}_{i}^{k+1}(k+1)}} \sin (\theta^{k+1}_j(t) - \theta^{k+1}_{i}(t))\right) \\
&+ \kappa\frac{1}{\bar{a}^{k+1}_{n-1}+1} \left( \underbrace{\underset{j \le n-1}{\sum_{j\in \mathcal{N}_{n-1}^{k+1}(k+1)}}\sin (\theta^{k+1}_j - \theta^{k+1}_{n-1})}_{\mathcal{I}_{12}} + \underbrace{\underset{j > n-1}{\sum_{j\in \mathcal{N}_{n-1}^{k+1}(k+1)}}\sin (\theta^{k+1}_j - \theta^{k+1}_{n-1})}_{\mathcal{I}_{13}}\right) \\
&+ \kappa\frac{1}{\bar{a}^{k+1}_{n-1}+1} \underbrace{ \sum_{l=0}^k \sum_{j\in \mathcal{N}_{n-1}^{k+1}(l)} \sin (\theta^{l}_j - \theta^{k+1}_{n-1})}_{\mathcal{I}_{14}} + \frac{\bar{a}^{k+1}_{n-1}}{\bar{a}^{k+1}_{n-1}+1}  \kappa S_k D_k(\theta(t)),
\end{aligned}
\end{equation}
where we use the fact
\begin{equation*}
\bar{a}^{k+1}_{n-1} = \eta(2N - n + 2)(\bar{a}^{k+1}_{n} +1 ), \quad 2N -n+2 = (N_{k+1} -n+2+S_k) + 2N - N_{k+1} - S_k.
\end{equation*}
$\ $

\noindent $\diamond$ \textbf{Estimates on $\mathbf{\mathcal{I}_{11}}$ in \eqref{F-8}.} We apply the strong connectivity of $\mathcal{G}_{k+1}$ and Lemma \ref{eta_sin_inequality} to obtain that
\begin{equation}\label{F-9}
\sum_{i=n}^{N_{k+1}}\left( \eta^{i-n} \underset{j \le i}{\min_{j\in \mathcal{N}_{i}^{k+1}(k+1)}} \sin (\theta^{k+1}_j(t) - \theta^{k+1}_{i}(t))\right) \le \sin(\theta^{k+1}_{\bar{k}_{n}} - \theta^{k+1}_{N_{k+1}}), 
\end{equation}
where $\bar{k}_{n} = \min_{j \in \bigcup_{i=n}^{N_{k+1}} \mathcal{N}^{k+1}_{i}(k+1)} j \le n-1$.
And it is obvious that $\mathcal{I}_{11} \le 0$.\newline

\noindent $\diamond$ \textbf{Estimates on $\mathbf{\mathcal{I}_{12}}$ in \eqref{F-8}.} For the term $\mathcal{I}_{12}$, we apply direct calculation to obtain that
\begin{equation}\label{F-d10}
\mathcal{I}_{12} = \underset{j \le n-1}{\sum_{j\in \mathcal{N}_{n-1}^{k+1}(k+1)}}\sin (\theta^{k+1}_j - \theta^{k+1}_{n-1}) \le \underset{j \le n-1}{\min_{j\in \mathcal{N}_{n-1}^{k+1}(k+1)}} \sin (\theta^{k+1}_j - \theta^{k+1}_{n-1}).
\end{equation}

\noindent $\diamond$ \textbf{Estimates on $\mathbf{\mathcal{I}_{13}}$ in \eqref{F-8}.} For the term $\mathcal{I}_{13}$, the estimate is almost the same as \eqref{C-8}. Without loss of generality, we only deal with $\mathcal{I}_{13}$ under the situation $\gamma > \frac{\pi}{2}$. According to \eqref{F-9}, we consider two cases depending on comparison between $\theta^{k+1}_{N_{k+1}} - \theta^{k+1}_{\bar{k}_{n}}$ and $\frac{\pi}{2}$. \newline

\noindent \textbf{(i)} For the first case that $ 0 \le \theta^{k+1}_{N_{k+1}} - \theta^{k+1}_{\bar{k}_{n}} \le \frac{\pi}{2}$, we immediately obtain that for $j \in \mathcal{N}^{k+1}_{n-1}(k+1), \ j>n-1$,
\begin{equation}\label{F-10}
0 \le \theta^{k+1}_{j}(t) - \theta^{k+1}_{n-1}(t) \le \theta^{k+1}_{N_{k+1}}(t) - \theta^{k+1}_{n-1}(t) \le \theta^{k+1}_{N_{k+1}}(t) - \theta^{k+1}_{\bar{k}_{n}}(t) \le \frac{\pi}{2}.
\end{equation}
Then it is known from \eqref{F-9}, \eqref{F-10} and $\eta > 2$ that
\begin{equation*}
\begin{aligned}
&\eta(N_{k+1} -  n + 1) \sum_{i=n}^{N_{k+1}}\left( \eta^{i-n} \underset{j \le i}{\min_{j\in \mathcal{N}_{i}^{k+1}(k+1)}} \sin (\theta^{k+1}_j(t) - \theta^{k+1}_{i}(t))\right) + \mathcal{I}_{13} \\
& \le \eta(N_{k+1} -  n + 1) \sin(\theta^{k+1}_{\bar{k}_{n}} - \theta^{k+1}_{N_{k+1}}) + \underset{j > n-1}{\sum_{j\in \mathcal{N}_{n-1}^{k+1}(k+1)}}\sin (\theta^{k+1}_j - \theta^{k+1}_{n-1}) \\
& \le (N_{k+1} -  n + 1) \sin(\theta^{k+1}_{\bar{k}_{n}} - \theta^{k+1}_{N_{k+1}}) + (N_{k+1} -  n + 1) \sin (\theta^{k+1}_{N_{k+1}} - \theta^{k+1}_{n-1}) \\
&\le 0.
\end{aligned}
\end{equation*}

\noindent \textbf{(ii)} For the second case that $ \frac{\pi}{2} < \theta^{k+1}_{N_{k+1}} - \theta^{k+1}_{\bar{k}_{n}} < \gamma$, it is known that
\begin{equation}\label{F-11}
\eta > \frac{1}{\sin \gamma} \quad \mbox{and} \quad \sin(\theta^{k+1}_{N_{k+1}} - \theta^{k+1}_{\bar{k}_{n}}) > \sin \gamma,
\end{equation}
which yields $\eta \sin(\theta^{k+1}_{\bar{k}_{n}} - \theta^{k+1}_{N_{k+1}} ) \le -1$. Thus we immediately derive that

\begin{equation*}
\begin{aligned}
&\eta(N_{k+1} -  n + 1) \sum_{i=n}^{N_{k+1}}\left( \eta^{i-n} \underset{j \le i}{\min_{j\in \mathcal{N}_{i}^{k+1}(k+1)}} \sin (\theta^{k+1}_j(t) - \theta^{k+1}_{i}(t))\right) + \mathcal{I}_{13} \\
& \le \eta(N_{k+1} -  n + 1) \sin(\theta^{k+1}_{\bar{k}_{n}} - \theta^{k+1}_{N_{k+1}}) + \underset{j > n-1}{\sum_{j\in \mathcal{N}_{n-1}^{k+1}(k+1)}}\sin (\theta^{k+1}_j - \theta^{k+1}_{n-1}) \\
& \le -(N_{k+1} -  n + 1) +(N_{k+1} -  n + 1) = 0.
\end{aligned}
\end{equation*}
Therefore, we combine the above arguments in (i) and (ii) to obtain
\begin{equation}\label{F-12}
\eta(N_{k+1} -  n + 1) \sum_{i=n}^{N_{k+1}}\left( \eta^{i-n} \underset{j \le i}{\min_{j\in \mathcal{N}_{i}^{k+1}(k+1)}} \sin (\theta^{k+1}_j - \theta^{k+1}_{i})\right) + \mathcal{I}_{13} \le 0.
\end{equation}

\noindent $\diamond$ \textbf{Estimates on $\mathbf{\mathcal{I}_{14}}$ in \eqref{F-8}.} For the term $\mathcal{I}_{14}$, there are three possible comparison between  $\theta^l_j$ with $0 \le l \le k$ and $\theta^{k+1}_{n-1}$:\newline

\noindent \textbf{(i)} If $\theta^l_j \le \theta^{k+1}_{n-1}$, we immediately have $\sin (\theta^{l}_j - \theta^{k+1}_{n-1}) \le 0$.\newline

\noindent \textbf{(ii)} If $ \theta^{k+1}_{n-1} < \theta^l_j \le \theta^{k+1}_{N_{k+1}}$, we consider two cases separately:\newline

\noindent (a) For the case that $ 0 \le \theta^{k+1}_{N_{k+1}} - \theta^{k+1}_{\bar{k}_{n}} \le \frac{\pi}{2}$, it is clear that
\[0 \le \theta^l_j - \theta^{k+1}_{n-1} \le \theta^{k+1}_{N_{k+1}} - \theta^{k+1}_{n-1}\le \theta^{k+1}_{N_{k+1}} - \theta^{k+1}_{\bar{k}_{n}} \le \frac{\pi}{2}.\]
Thus from the above inequality and \eqref{F-9}, we have
\begin{equation*}
\begin{aligned}
&\eta \sum_{i=n}^{N_{k+1}}\left( \eta^{i-n} \underset{j \le i}{\min_{j\in \mathcal{N}_{i}^{k+1}(k+1)}} \sin (\theta^{k+1}_j(t) - \theta^{k+1}_{i}(t))\right) + \sin (\theta^l_j - \theta^{k+1}_{n-1}) \\
&\le \eta \sin(\theta^{k+1}_{\bar{k}_{n}} - \theta^{k+1}_{N_{k+1}}) + \sin (\theta^l_j - \theta^{k+1}_{n-1}) \\
&\le \sin(\theta^{k+1}_{\bar{k}_{n}} - \theta^{k+1}_{N_{k+1}}) + \sin(\theta^{k+1}_{N_{k+1}} - \theta^{k+1}_{\bar{k}_{n}}) = 0.
\end{aligned} 
\end{equation*}
\noindent (b) For another case that $ \frac{\pi}{2} < \theta^{k+1}_{N_{k+1}} - \theta^{k+1}_{\bar{k}_{n}} < \gamma$, it is known from \eqref{F-11} that
\begin{equation*}
\begin{aligned}
&\eta \sum_{i=n}^{N_{k+1}}\left( \eta^{i-n} \underset{j \le i}{\min_{j\in \mathcal{N}_{i}^{k+1}(k+1)}} \sin (\theta^{k+1}_j(t) - \theta^{k+1}_{i}(t))\right) + \sin (\theta^l_j - \theta^{k+1}_{n-1}) \\
&\le \eta \sin(\theta^{k+1}_{\bar{k}_{n}} - \theta^{k+1}_{N_{k+1}}) + \sin (\theta^l_j - \theta^{k+1}_{n-1}) \\
&\le -1 + 1 = 0
\end{aligned} 
\end{equation*}
Hence, combining the above arguments in (a) and (b), we obtain that
\[\eta \sum_{i=n}^{N_{k+1}}\left( \eta^{i-n} \underset{j \le i}{\min_{j\in \mathcal{N}_{i}^{k+1}(k+1)}} \sin (\theta^{k+1}_j(t) - \theta^{k+1}_{i}(t))\right) + \sin (\theta^l_j - \theta^{k+1}_{n-1}) \le 0.\]
$\ $

\noindent \textbf{(iii)} If $\theta^l_j > \theta^{k+1}_{N_{k+1}}$, we exploit the concave property of sine function in $[0, \pi]$ to get
\begin{equation}\label{F-13}
\sin (\theta^l_j - \theta^{k+1}_{n-1}) \le \sin (\theta^l_j - \theta^{k+1}_{N_{k+1}}) + \sin (\theta^{k+1}_{N_{k+1}} - \theta^{k+1}_{n-1}).
\end{equation}
For the second part on the right-hand side of above inequality \eqref{F-13}, we apply the same analysis in (ii) to obtain
\[\eta \sum_{i=n}^{N_{k+1}}\left( \eta^{i-n} \underset{j \le i}{\min_{j\in \mathcal{N}_{i}^{k+1}(k+1)}} \sin (\theta^{k+1}_j - \theta^{k+1}_{i})\right) + \sin (\theta^{k+1}_{N_{k+1}} - \theta^{k+1}_{n-1}) \le 0.\]
For the first part on the right-hand side of \eqref{F-13}, the calculation is the same as the formula \eqref{F-d8}, thus we have
\begin{equation*}
\sin (\theta^{l}_j - \theta^{k+1}_{N_{k+1}}) \le \theta^{l}_j - \theta^{k+1}_{N_{k+1}} \le \theta^{l}_j -  \min_{0 \le i \le k}\min_{1 \le j \le N_i} \{\theta^i_j\} \le D_k(\theta(t)).
\end{equation*}
Therefore, we combine the above estimates to obtain
\begin{equation}\label{F-14}
\begin{aligned}
&\eta S_k \sum_{i=n}^{N_{k+1}}\left( \eta^{i-n} \underset{j \le i}{\min_{j\in \mathcal{N}_{i}^{k+1}(k+1)}} \sin (\theta^{k+1}_j(t) - \theta^{k+1}_{i}(t))\right) + \mathcal{I}_{14} \\
& \le \eta S_k \sin(\theta^{k+1}_{\bar{k}_{n}} - \theta^{k+1}_{N_{k+1}}) + \sum_{l=0}^k \sum_{j\in \mathcal{N}_{n-1}^{k+1}(l)} \sin (\theta^{l}_j - \theta^{k+1}_{n-1}) \\
& \le S_k D_k(\theta(t)).
\end{aligned}
\end{equation}
Then combining \eqref{F-d10}, \eqref{F-12}, \eqref{F-14} and coming back to \eqref{F-8}, we obtain that
\begin{equation*}
\begin{aligned}
\frac{d}{dt} \bar{\theta}^{k+1}_{n-1} &\le \Omega_M + \kappa \frac{\eta}{\bar{a}^{k+1}_{n-1}+1} \sum_{i=n}^{N_{k+1}}\left( \eta^{i-n} \underset{j \le i}{\min_{j\in \mathcal{N}_{i}^{k+1}(k+1)}} \sin (\theta^{k+1}_j - \theta^{k+1}_{i})\right)  \\
&+ \kappa\frac{1}{\bar{a}^{k+1}_{n-1}+1} \underset{j \le n-1}{\min_{j\in \mathcal{N}_{n-1}^{k+1}(k+1)}} \sin (\theta^{k+1}_j - \theta^{k+1}_{n-1}) \\
&+ \frac{\bar{a}^{k+1}_{n-1}}{\bar{a}^{k+1}_{n-1}+1}  \kappa S_k D_k(\theta(t)) + \frac{1}{\bar{a}^{k+1}_{n-1}+1}  \kappa S_k D_k(\theta(t))\\
&=\Omega_M + \kappa \frac{1}{\bar{a}^{k+1}_{n-1}+1} \sum_{i=n-1}^{N_{k+1}}\left( \eta^{i-(n-1)} \underset{j \le i}{\min_{j\in \mathcal{N}_{i}^{k+1}(k+1)}} \sin (\theta^{k+1}_j - \theta^{k+1}_{i})\right) + \kappa S_k D_k(\theta(t)).
\end{aligned}
\end{equation*}
This means that the claim \eqref{F-5} does hold for $n-1$. Therefore, we apply the inductive criteria to complete the proof of the claim \eqref{F-5}.\newline 

\noindent $\bigstar$ \textbf{Step 2.} Now we are ready to prove \eqref{F-1} on $J_l$ for Case 2. In fact, we apply Lemma \ref{eta_sin_inequality} and the strong connectivity of $\mathcal{G}_{k+1}$ to have
\begin{equation*}
\sum_{i=1}^{N_{k+1}}\left( \eta^{i-1} \underset{j \le i}{\min_{j\in \mathcal{N}_{i}^{k+1}(k+1)}} \sin (\theta^{k+1}_j - \theta^{k+1}_{i})\right) \le \sin(\theta^{k+1}_1 - \theta^{k+1}_{N_{k+1}})
\end{equation*}
From the notations in \eqref{abbreviation_k} and \eqref{bar_underline_k}, it is known that
\begin{equation*}
\bar{\theta}_1^{k+1} = \bar{\theta}_{k+1}, \quad \underline{\theta}_{N_{k+1}}^{k+1} = \underline{\theta}_{k+1}.
\end{equation*}
Thus, we exploit the above inequality and set $n=1$ in \eqref{F-5} to obtain
\begin{equation}\label{F-15}
\begin{aligned}
\frac{d}{dt} \bar{\theta}_{k+1} &=\frac{d}{dt} \bar{\theta}^{k+1}_{1} \\
&\le \Omega_M + \kappa \frac{1}{\bar{a}^{k+1}_{1}+1} \sum_{i=1}^{N_{k+1}}\left( \eta^{i-1} \underset{j \le i}{\min_{j\in \mathcal{N}_{i}^{k+1}(k+1)}} \sin (\theta^{k+1}_j - \theta^{k+1}_{i})\right) + \kappa S_k D_k(\theta(t))\\
&\le \Omega_M + \kappa S_k D_k(\theta(t)) + \kappa \frac{1}{\bar{a}^{k+1}_{1}+1} \sin(\theta^{k+1}_1 - \theta^{k+1}_{N_{k+1}})
\end{aligned}
\end{equation}
We further apply the similar arguments in obtaining the dynamics of $\bar{\theta}_{k+1}$ in \eqref{F-15} to derive the differential inequality of $\underline{\theta}_{k+1}$ as below
\begin{equation}\label{F-16}
\frac{d}{dt} \underline{\theta}_{k+1} \ge \Omega_m + \kappa \frac{1}{\bar{a}^{k+1}_{1}+1} \sin(\theta^{k+1}_{N_{k+1}} - \theta^{k+1}_1) - \kappa S_k D_k(\theta(t)).
\end{equation}
Due to the monotone decreasing property of $\frac{\sin x}{x}$ in $(0, \pi]$ and from \eqref{F-4}, it is obvious that
\[\sin (\theta^{k+1}_{N_{k+1}} - \theta^{k+1}_1) \ge \frac{\sin \gamma}{\gamma}(\theta^{k+1}_{N_{k+1}} - \theta^{k+1}_{1}).\]
Then we combine the above inequality, \eqref{F-15}, \eqref{F-16} and \eqref{a^k_1-size} to get
\begin{equation*}
\begin{aligned}
\dot{Q}^{k+1}(t) &= \frac{d}{dt} (\bar{\theta}_{k+1} - \underline{\theta}_{k+1}) \le D(\Omega) - \kappa \frac{2}{\bar{a}^{k+1}_{1}+1} \sin(\theta^{k+1}_{N_{k+1}} - \theta^{k+1}_1) + 2\kappa S_k D_k(\theta(t))\\
&\le D(\Omega) - \kappa \frac{1}{\bar{a}^{k+1}_{1}+1} \frac{\sin \gamma}{\gamma}(\theta^{k+1}_{N_{k+1}} - \theta^{k+1}_1) + 2\kappa S_k D_k(\theta(t)) \\
&\le D(\Omega) - \kappa \frac{1}{\bar{a}^{k+1}_{1}+1} \frac{\sin \gamma}{\gamma} Q^{k+1}(t)+ 2\kappa S_k D_k(\theta(t)) \\
&\le D(\Omega) - \kappa \frac{1}{\sum_{j=1}^{N-1} (\eta^j A(2N,j)) + 1} \frac{\sin \gamma}{\gamma} Q^{k+1}(t)+ \kappa (2N+1) D_k(\theta(t)), \quad t\in J_l,
\end{aligned}
\end{equation*}
where we use the fact that $Q^{k+1}(t) \le \theta^{k+1}_{N_{k+1}}(t) - \theta^{k+1}_1(t)$. Thus we obtain the dynamics for $Q^{k+1}(t)$ in \eqref{F-1} on $J_l$.\newline

\subsection{Case 3} Consider the case that
\begin{equation*}
\max_{0 \le i \le k+1}\{\bar{\theta}_i\} = \bar{\theta}_{k+1},\quad \min_{0 \le i \le k+1}\{\underline{\theta}_i\} = \min_{0 \le i \le k}\{\underline{\theta}_i\} \quad \mbox{on} \ J_l.
\end{equation*}
\begin{figure}[h]
\centering
\includegraphics[width=0.5\textwidth]{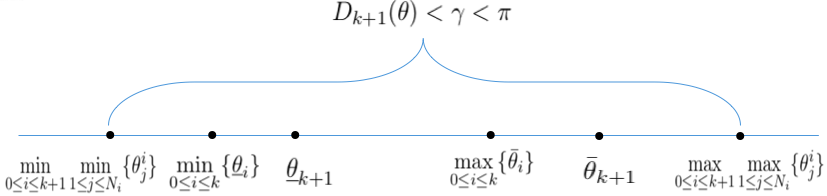}
\caption{The comparison relation in Case 3}
\label{Fig3}
\end{figure}
The comparison relation in this case is presented in Figure \ref{Fig3}.
For this case, without loss of generality, we set
\begin{equation*}
\underline{\theta}_q = \min_{0 \le i \le k}\{\underline{\theta}_i\} \quad \mbox{where} \ 0\le q \le k. 
\end{equation*}
We further assume
\begin{equation*}
\theta^{k+1}_1 \le \theta^{k+1}_2 \le \ldots \le \theta^{k+1}_{N_{k+1}}, \quad \theta^{q}_1 \le \theta^{q}_2 \le \ldots \le \theta^{q}_{N_{q}} \quad \mbox{in} \ J_l.
\end{equation*}
It is obvious that $\underline{\theta}_q = \min_{0 \le i \le q}\{\underline{\theta}_i\}$. Thus we apply the same arguments in Case 2 to obtain
\begin{equation}\label{F-17}
\frac{d}{dt} \underline{\theta}_{q} \ge \Omega_m + \kappa \frac{1}{\bar{a}^{q}_{1}+1} \sin(\theta^{q}_{N_{q}} - \theta^{q}_1) - \kappa S_{q-1} D_{q-1}(\theta(t)).
\end{equation}
In the subsequence, we prove \eqref{F-1} on $J_l$ in two sub-cases depending on the comparison between $\theta^{k+1}_1 $ and $\max\limits_{0 \le i \le k}\max\limits_{1 \le j \le N_i}\{\theta^i_j\}$. \newline

\noindent $\bullet$ \textbf{Case 3.1.} If $\theta^{k+1}_1  \le \max\limits_{0 \le i \le k}\max\limits_{1 \le j \le N_i}\{\theta^i_j\}$, then we combine \eqref{F-4}, \eqref{F-15} and \eqref{F-17} to get
\begin{equation*}
\begin{aligned}
\dot{Q}^{k+1}(t) &= \frac{d}{dt} (\bar{\theta}_{k+1} - \underline{\theta}_q) \\
&\le D(\Omega) + \kappa S_k D_k(\theta(t)) - \kappa \frac{1}{\bar{a}^{k+1}_{1}+1} \sin(\theta^{k+1}_{N_{k+1}} - \theta^{k+1}_1) \\
&- \kappa \frac{1}{\bar{a}^{q}_{1}+1} \sin(\theta^{q}_{N_{q}} - \theta^{q}_1) + \kappa S_{q-1} D_{q-1}(\theta(t))\\
&\le D(\Omega) - \kappa\min \left\{\frac{1}{\bar{a}^{k+1}_{1}+1}, \frac{1}{\bar{a}^{q}_{1}+1}\right\} \left( \sin(\theta^{k+1}_{N_{k+1}} - \theta^{k+1}_1) + \sin(\theta^{q}_{N_{q}} - \theta^{q}_1)\right) \\
&+ \kappa (S_k + S_{q-1}) D_k(\theta(t)) \\
& \le D(\Omega) - \kappa\min \left\{\frac{1}{\bar{a}^{k+1}_{1}+1}, \frac{1}{\bar{a}^{q}_{1}+1}\right\} \frac{\sin \gamma}{\gamma} \left(\theta^{k+1}_{N_{k+1}} - \theta^{k+1}_1 + \theta^{q}_{N_{q}} - \theta^{q}_1\right) \\
&+ \kappa (S_k + S_{q-1}) D_k(\theta(t))- \kappa\min \left\{\frac{1}{\bar{a}^{k+1}_{1}+1}, \frac{1}{\bar{a}^{q}_{1}+1}\right\} \frac{\sin \gamma}{\gamma}  \left( \max\limits_{0 \le i \le k}\max\limits_{1 \le j \le N_i}\{\theta^i_j\} - \theta^q_{N_q} \right)\\
&+ \kappa\min \left\{\frac{1}{\bar{a}^{k+1}_{1}+1}, \frac{1}{\bar{a}^{q}_{1}+1}\right\} \frac{\sin \gamma}{\gamma}  \left( \max\limits_{0 \le i \le k}\max\limits_{1 \le j \le N_i}\{\theta^i_j\} - \theta^q_{N_q} \right) \\
&\le D(\Omega) - \kappa\min \left\{\frac{1}{\bar{a}^{k+1}_{1}+1}, \frac{1}{\bar{a}^{q}_{1}+1}\right\} \frac{\sin \gamma}{\gamma} \left( \theta^{k+1}_{N_{k+1}} - \theta^q_1\right) + \kappa (S_k + S_{q-1}) D_k(\theta(t)) \\
&+ \kappa\min \left\{\frac{1}{\bar{a}^{k+1}_{1}+1}, \frac{1}{\bar{a}^{q}_{1}+1}\right\} \frac{\sin \gamma}{\gamma} D_k(\theta(t)) \\
&\le D(\Omega) - \kappa \frac{1}{\sum_{j=1}^{N-1} (\eta^j A(2N,j)) + 1} \frac{\sin \gamma}{\gamma} Q^{k+1}(t)+ \kappa (2N+1) D_k(\theta(t)), \qquad \mbox{in} \ J_l.
\end{aligned}
\end{equation*}
In above estimates, we apply \eqref{a^k_1-size}, \eqref{F-d7}  and the fact that
\begin{equation*}
Q^{k+1}(t) = \bar{\theta}_{k+1} - \underline{\theta}_q \le \theta^{k+1}_{N_{k+1}} - \theta^q_1 \quad \mbox{and} \quad \min \left\{\frac{1}{\bar{a}^{k+1}_{1}+1}, \frac{1}{\bar{a}^{q}_{1}+1}\right\} \frac{\sin \gamma}{\gamma} \le 1.
\end{equation*}

\noindent $\bullet$ \textbf{Case 3.2.} If $\theta^{k+1}_1  > \max\limits_{0 \le i \le k}\max\limits_{1 \le j \le N_i}\{\theta^i_j\}$, similar to Case 2, we can apply the induction principle to prove that for $1\le n \le N_{k+1}$,
\begin{equation}\label{F-18}
\begin{aligned}
\frac{d}{dt} \bar{\theta}^{k+1}_n &\le \Omega_M + \kappa \frac{1}{\bar{a}^{k+1}_{n}+1} \sum_{i=n}^{N_{k+1}}\left( \eta^{i-n} \underset{j \le i}{\min_{j\in \mathcal{N}_{i}^{k+1}(k+1)}} \sin (\theta^{k+1}_j - \theta^{k+1}_{i})\right) \\
&+\kappa \frac{1}{\bar{a}^{k+1}_{n}+1}\sum_{l=0}^k \sum_{j\in \mathcal{N}_{n}^{k+1}(l)} \sin (\theta^{l}_j - \theta^{k+1}_{n})\\
&+ \kappa \sum_{i=n+1}^{N_{k+1}} \left( \frac{\prod_{r=n}^{i-1} \bar{a}^{k+1}_r}{\prod_{r=n}^i (\bar{a}^{k+1}_r + 1)} \sum_{l=0}^k \sum_{j\in \mathcal{N}_{i}^{k+1}(l)} \sin (\theta^{l}_j - \theta^{k+1}_{i})\right).
\end{aligned}
\end{equation}
Since the proof of \eqref{F-18} is similar to that of \eqref{F-5}, we omit its details.
In particular, we set $n=1$ in the above inequality \eqref{F-18} and apply Lemma \ref{eta_sin_inequality} to get
\begin{equation}\label{F-a18}
\begin{aligned}
\frac{d}{dt} \bar{\theta}_{k+1}  = \frac{d}{dt} \bar{\theta}^{k+1}_1 &\le \Omega_M + \kappa \frac{1}{\bar{a}^{k+1}_{1}+1} \sin (\theta^{k+1}_1 - \theta^{k+1}_{N_{k+1}}) + \kappa \frac{1}{\bar{a}^{k+1}_{1}+1}\sum_{l=0}^k \sum_{j\in \mathcal{N}_{1}^{k+1}(l)} \sin (\theta^{l}_j - \theta^{k+1}_{1}) \\
&+ \kappa \sum_{i=2}^{N_{k+1}} \left( \frac{\prod_{r=1}^{i-1} \bar{a}^{k+1}_r}{\prod_{r=1}^i (\bar{a}^{k+1}_r + 1)} \sum_{l=0}^k \sum_{j\in \mathcal{N}_{i}^{k+1}(l)} \sin (\theta^{l}_j - \theta^{k+1}_{i})\right)
\end{aligned}
\end{equation}
Due to the situation that $\theta^{k+1}_1  > \max\limits_{0 \le i \le k}\max\limits_{1 \le j \le N_i}\{\theta^i_j\}$, it is known that for $0 \le l \le k$, the term $\sin (\theta^l_j - \theta^{k+1}_i)$ in \eqref{F-a18} is non-positive. And according to the spanning tree structure, the neighbors set of $\mathcal{G}_{k+1}$ in  $\bigcup_{l=0}^k\mathcal{G}_l$ is non-empty,
\begin{equation*}
\bigcup_{i=1}^{N_{k+1}} \bigcup_{l=0}^k  \mathcal{N}_i^{k+1}(l) \ne \emptyset,
\end{equation*}
this means that there must exist some $\theta^l_n$ belonging to $\bigcup_{l=0}^k\mathcal{G}_l$ and $\theta^{k+1}_m$ such that $\theta^l_n \in \mathcal{N}^{k+1}_m (l)$. Moreover, from \eqref{coeffi_k}, it is clear that for the coefficients in the last term in \eqref{F-a18} satisfy
\begin{equation*}
\frac{\prod_{r=1}^{i-1} \bar{a}^{k+1}_r}{\prod_{r=2}^i (\bar{a}^{k+1}_r + 1)} = \frac{\prod_{r=1}^{i-1} \bar{a}^{k+1}_r}{\prod_{r=1}^{i-1} (\bar{a}^{k+1}_{r+1} + 1)} = \prod_{r=1}^{i-1} \eta (2N- r+1) > 1 \quad \mbox{with} \  2 \le i \le N_{k+1}.
\end{equation*}
Then we combine the above estimates to have
\begin{equation}\label{F-19}
\begin{aligned}
\frac{d}{dt} \bar{\theta}_{k+1} &\le \Omega_M + \kappa \frac{1}{\bar{a}^{k+1}_{1}+1} \sin (\theta^{k+1}_1 - \theta^{k+1}_{N_{k+1}}) + \kappa \frac{1}{\bar{a}^{k+1}_{1}+1} \sum_{l=0}^k \sum_{j\in \mathcal{N}_{1}^{k+1}(l)} \sin (\theta^{l}_j - \theta^{k+1}_{1})\\
& +  \kappa \frac{1}{\bar{a}^{k+1}_{1}+1}\sum_{i=2}^{N_{k+1}} \left( \frac{\prod_{l=1}^{i-1} \bar{a}^{k+1}_l}{\prod_{l=2}^i (\bar{a}^{k+1}_l + 1)} \sum_{l=0}^k \sum_{j\in \mathcal{N}_{i}^{k+1}(l)} \sin (\theta^{l}_j - \theta^{k+1}_{i})\right)  \\
& \le \Omega_M + \kappa \frac{1}{\bar{a}^{k+1}_{1}+1} \sin (\theta^{k+1}_1 - \theta^{k+1}_{N_{k+1}}) + \kappa \frac{1}{\bar{a}^{k+1}_{1}+1} \sin (\theta^{l}_n - \theta^{k+1}_{m})\\
& \le \Omega_M - \kappa \frac{1}{\bar{a}^{k+1}_{1}+1} \frac{\sin \gamma}{\gamma} \left( \theta^{k+1}_{N_{k+1}} - \theta^{k+1}_1 \right )  - \kappa \frac{1}{\bar{a}^{k+1}_{1}+1} \frac{\sin \gamma}{\gamma} \left( \theta^{k+1}_{m}- \theta^{l}_n\right) \\
& \le \Omega_M - \kappa \frac{1}{\bar{a}^{k+1}_{1}+1} \frac{\sin \gamma}{\gamma} \left(\theta^{k+1}_{N_{k+1}} - \theta^{k+1}_1 + \theta^{k+1}_1 - \max\limits_{0 \le i \le k}\max\limits_{1 \le j \le N_i}\{\theta^i_j\} \right) \\
& = \Omega_M - \kappa \frac{1}{\bar{a}^{k+1}_{1}+1} \frac{\sin \gamma}{\gamma} \left(\theta^{k+1}_{N_{k+1}} - \max\limits_{0 \le i \le k}\max\limits_{1 \le j \le N_i}\{\theta^i_j\} \right),
\end{aligned}
\end{equation}
where we exploit the property
\[\theta^{k+1}_{m}- \theta^{l}_n \ge \theta^{k+1}_1 - \max\limits_{0 \le i \le k}\max\limits_{1 \le j \le N_i}\{\theta^i_j\}.\]
Then we combine \eqref{F-17} and \eqref{F-19} to obtain that
\begin{equation*}
\begin{aligned}
\dot{Q}^{k+1}(t) &\le D(\Omega) -\kappa \frac{1}{\bar{a}^{k+1}_{1}+1} \frac{\sin \gamma}{\gamma} \left(\theta^{k+1}_{N_{k+1}} - \max\limits_{0 \le i \le k}\max\limits_{1 \le j \le N_i}\{\theta^i_j\} \right) \\
&- \kappa \frac{1}{\bar{a}^{q}_{1}+1} \sin(\theta^{q}_{N_{q}} - \theta^{q}_1) + \kappa S_{q-1} D_{q-1}(\theta(t)) \\
&\le D(\Omega) - \kappa \min \left\{\frac{1}{\bar{a}^{k+1}_{1}+1}, \frac{1}{\bar{a}^{q}_{1}+1}\right\} \frac{\sin \gamma}{\gamma} \left(\theta^{k+1}_{N_{k+1}} - \max\limits_{0 \le i \le k}\max\limits_{1 \le j \le N_i}\{\theta^i_j\} + \theta^{q}_{N_{q}} - \theta^{q}_1\right) \\
&- \kappa \min \left\{\frac{1}{\bar{a}^{k+1}_{1}+1}, \frac{1}{\bar{a}^{q}_{1}+1}\right\} \frac{\sin \gamma}{\gamma} \left(\max\limits_{0 \le i \le k}\max\limits_{1 \le j \le N_i}\{\theta^i_j\} - \theta^q_{N_q} \right) \\
&+ \kappa \min \left\{\frac{1}{\bar{a}^{k+1}_{1}+1}, \frac{1}{\bar{a}^{q}_{1}+1}\right\} \frac{\sin \gamma}{\gamma} \left(\max\limits_{0 \le i \le k}\max\limits_{1 \le j \le N_i}\{\theta^i_j\} - \theta^q_{N_q} \right) + \kappa S_{q-1} D_{q-1}(\theta(t)) \\
&\le D(\Omega) - \kappa \min \left\{\frac{1}{\bar{a}^{k+1}_{1}+1}, \frac{1}{\bar{a}^{q}_{1}+1}\right\} \frac{\sin \gamma}{\gamma} \left(\theta^{k+1}_{N_{k+1}} - \theta^q_1\right) + \kappa(2N+1)D_k(\theta(t)) \\
&\le D(\Omega) - \kappa \frac{1}{\sum_{j=1}^{N-1} (\eta^j A(2N,j)) + 1} \frac{\sin \gamma}{\gamma} Q^{k+1}(t)+ \kappa (2N+1) D_k(\theta(t)), \qquad \mbox{in} \ J_l.
\end{aligned}
\end{equation*}
$\ $

\subsection{Case 4} Consider the case that
\begin{equation*}
\max_{0 \le i \le k+1}\{\bar{\theta}_i\} = \max_{0 \le i \le k}\{\bar{\theta}_i\},\quad \min_{0 \le i \le k+1}\{\underline{\theta}_i\} = \underline{\theta}_{k+1} \quad \mbox{in} \ J_l.
\end{equation*}
\begin{figure}[h]
\centering
\includegraphics[width=0.5\textwidth]{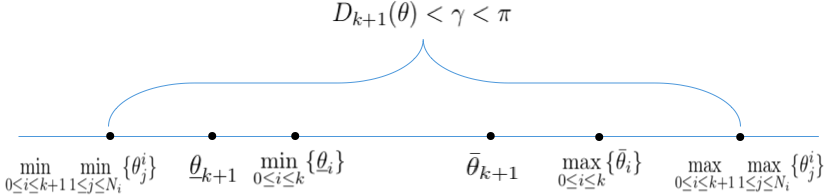}
\caption{The comparison relation in Case 4}
\label{Fig4}
\end{figure}
The comparison relation in this case is showed in Figure \ref{Fig4}.
For this case, the analysis is similar to that in Case 3. Therefore, we omit the details of discussion. \newline

\subsection{Conclusion} Since all analysis above do not depend on interval $J_l$ with $1\le l \le r$, thus we combine all analysis in Case 1, Case 2, Case 3, and Case 4 to derive that 
\begin{equation*}
\dot{Q}^{k+1}(t) \le D(\Omega) - \kappa \frac{1}{\sum_{j=1}^{N-1} (\eta^j A(2N,j)) + 1} \frac{\sin \gamma}{\gamma} Q^{k+1}(t)+ \kappa (2N+1) D_k(\theta(t)), \qquad \mbox{in} \ [0, T^*).
\end{equation*}
\qed

\end{appendix}

\end{document}